\let\ORIlabel\label
\let\ORIrefstepcounter\refstepcounter
\AddToHook{package/hyperref/before}{
   \let\label\ORIlabel 
   \let\refstepcounter\ORIrefstepcounter}

\documentclass[final,onefignum,onetabnum]{siamart220329}
\usepackage{shuffle}

\usepackage[a4paper, left = 1in, right = 1in, bottom = 1in, top = 1in]{geometry}

\usepackage[nocompress]{cite}
\usepackage{lipsum}
\usepackage{amsfonts}
\usepackage{graphicx}
\usepackage{epstopdf}
\usepackage{algorithmic}
\usepackage{mathtools}
\usepackage{amsmath,amssymb}
\ifpdf
  \DeclareGraphicsExtensions{.eps,.pdf,.png,.jpg}
\else
  \DeclareGraphicsExtensions{.eps}
\fi

\usepackage{subcaption}
\usepackage{dirtytalk}
\usepackage{tikz}
\usetikzlibrary{arrows}
\usepackage{multirow}
\newcommand{\ord}{\mathrm{ord}}

\newsiamremark{remark}{Remark}
\newsiamremark{notation}{Notation}
\newsiamremark{example}{Example}
\newsiamremark{hypothesis}{Hypothesis}
\crefname{hypothesis}{Hypothesis}{Hypotheses}
\newsiamthm{claim}{Claim}
\newsiamthm{conjecture}{Conjecture}

\newcommand{\theoremnum}{}
\newtheorem*{theorem**}{Theorem\theoremnum}

\newcommand{\propositionnum}{}
\newtheorem*{proposition**}{Proposition\propositionnum}

\newcommand{\corollarynum}{}
\newtheorem*{corollary**}{Corollary\corollarynum}

\usepackage{tikz}
\usetikzlibrary{arrows,calc,positioning}

\makeatletter
\newcommand\RSloop{\@ifnextchar\bgroup\RSloopa\RSloopb}
\makeatother
\newcommand\RSloopa[1]{\bgroup\RSloop#1\relax\egroup\RSloop}
\newcommand\RSloopb[1]%
  {\ifx\relax#1%
   \else
     \ifcsname RS:#1\endcsname
       \csname RS:#1\endcsname
     \else
       \GenericError{(RS)}{RS Error: operator #1 undefined}{}{}%
     \fi
   \expandafter\RSloop
   \fi
  }
\newcommand\X{0}
\newcommand\RS[1]%
  {\begin{tikzpicture}
     [every node/.style=
       {circle,draw,fill,minimum size=3pt,inner sep=0pt,outer sep=0pt},
      line cap=round
     ]
   \coordinate(\X) at (0,0);
   \RSloop{#1}\relax
   \end{tikzpicture}
  }

\newcommand\RSdef[1]{\expandafter\def\csname RS:#1\endcsname}
\newlength\RSu
\RSdef{n}{\draw (\X) node{};}
\RSdef{i}{\draw (\X) node{} -- + (90:\RSu) node{};}
\RSdef{d}{\draw[dashed] (\X) node{} -- + (90:\RSu) node{};}
\RSdef{l}{\draw (\X) node{} -- +(135:\RSu) node{};}
\RSdef{r}{\draw (\X) node{} -- +(45:\RSu) node{};}
\RSdef{I}{\draw (\X) node{} -- +(90:\RSu) coordinate(\X I);\edef\X{\X I}}
\RSdef{L}{\draw (\X) node{} -- +(135:\RSu) coordinate(\X L);\edef\X{\X L}}
\RSdef{R}{\draw (\X) node{} -- +(45:\RSu) coordinate(\X R);\edef\X{\X R}}

\RSdef{t}<1>{\draw (\X) node{};}

\newcommand{\commonRootTree}{\begin{tikzpicture}
        \coordinate (root) at (0, 0);
        \filldraw (root) circle [radius=1.5pt];
        
        \coordinate (n1) at ($(root) + (159:4.3ex)$);
        \node[] at ($(n1) + (0,1ex)$) {\footnotesize $\tau_1$};
        
        \coordinate (n2) at ($(root) + (130:2ex)$);
        \node[] at ($(n2) + (0,1ex)$) {\footnotesize $\tau_2$};

        \coordinate (n3) at ($(root) + (50:2ex)$);
        \node[] at ($(n3) + (0,1ex)$) {\footnotesize $\cdots$};

        \coordinate (n4) at ($(root) + (22:4.1ex)$);
        \node[] at ($(n4) + (0,1ex)$) {\footnotesize $\tau_k$};
        
        \draw [thin] (n1) to (root);
        \draw [thin] (n2) to (root);
        \draw [thin] (n4) to (root);
    \end{tikzpicture}}

\newcommand{\treeFour}{
\begin{tikzpicture}
\coordinate (root) at (0, 0);
\filldraw (root) circle [radius=1.5pt];

\coordinate (left) at ($(root) + (130:5ex)$);
\filldraw (left) circle [radius=1.5pt];

\coordinate (right) at ($(root) + (50:5ex)$);
\filldraw (right) circle [radius=1.5pt];

\coordinate (rl) at ($(right) + (130:5ex)$);
\filldraw (rl) circle [radius=1.5pt];

\coordinate (rr) at ($(right) + (50:5ex)$);
\filldraw (rr) circle [radius=1.5pt];

\draw [thin] (left) to (root) to (right);
\draw [thin] (rl) to (right) to (rr);
\draw [thin] (rl) to (right) to (rr);

\node[] at ($(root) + (0ex,-2ex)$) {\small $b_i$};
\node[] at ($(left) + (-0.1ex,-2.5ex)$) {\small $a_{ik}$};
\node[] at ($(right) + (0.1ex,-2.5ex)$) {\small $a_{ij}$};
\node[] at ($(rl) + (-0.1ex,-2.5ex)$) {\small $a_{j\ell}$};
\node[] at ($(rr) + (0.6ex,-2.5ex)$) {\small $a_{jm}$};
\end{tikzpicture}}

\newcommand{\treeFive}{
\begin{tikzpicture}
\coordinate (root) at (0, 0);
\filldraw (root) circle [radius=1.5pt];

\coordinate (left) at ($(root) + (130:5ex)$);
\filldraw (left) circle [radius=1.5pt];

\coordinate (right) at ($(root) + (50:5ex)$);
\filldraw (right) circle [radius=1.5pt];

\coordinate (rl) at ($(right) + (130:5ex)$);
\filldraw (rl) circle [radius=1.5pt];

\coordinate (rr) at ($(right) + (50:5ex)$);
\filldraw (rr) circle [radius=1.5pt];

\draw [thin] (left) to (root);
\draw [thin] (rl) to (right) to (rr);
\draw [thin] (rl) to (right) to (rr);

\node[] at ($(root) + (0ex,-2ex)$) {\small $b_i$};
\node[] at ($(left) + (-0.1ex,-2.5ex)$) {\small $a_{ik}$};
\node[] at ($(right) + (0ex,-2ex)$) {\small $b_j$};
\node[] at ($(rl) + (-0.1ex,-2.5ex)$) {\small $a_{j\ell}$};
\node[] at ($(rr) + (0.6ex,-2.5ex)$) {\small $a_{jm}$};
\end{tikzpicture}}

\mathchardef\mhyphen="2D

\headers{Explicit and Effectively Symmetric Runge--Kutta Methods}{D. Shmelev, K. Ebrahimi-Fard, N. Tapia, and C. Salvi}

\title{Explicit and Effectively Symmetric Runge--Kutta Methods}

\author{
Daniil Shmelev\thanks{Department of Mathematics, Imperial College London, London, United Kingdom (\email{daniil.shmelev23@imperial.ac.uk}, \email{c.salvi@imperial.ac.uk}). Corresponding author: D. Shmelev.}
\and Kurusch Ebrahimi-Fard\thanks{Department of Mathematical Sciences, Norwegian University of Science and Technology, Trondheim, Norway (\email{kurusch.ebrahimi-fard@ntnu.no}).}
\and Nikolas Tapia\thanks{Weierstrass Institute for Applied Analysis and Stochastics \& Institute for Mathematics, Humboldt University of Berlin, Berlin, Germany (\email{tapia@wias-berlin.de}).}
\and Cristopher Salvi\footnotemark[1]
}

\usepackage{amsopn}

\usepackage{bookmark}

\ifpdf
\hypersetup{
  pdftitle={Explicit and Effectively Symmetric Runge--Kutta Methods},
  pdfauthor={D. Shmelev, K.~Ebrahimi-Fard, N.~Tapia, and C.~Salvi}
}
\fi

\begin{document}

\maketitle

\begin{abstract}
Symmetry is a key property of numerical methods. The geometric properties of symmetric schemes make them an attractive option for integrating Hamiltonian systems, whilst their ability to exactly recover the initial condition without the need to store the entire solution trajectory makes them ideal for the efficient implementation of Neural ODEs. In this work, we employ a Hopf algebraic approach to the study of symmetric B‑series methods. We show that every B-series method can be written as the composition of a symmetric and \say{antisymmetric} component, and explore the structure of this decomposition for Runge--Kutta schemes. A major bottleneck of symmetric Runge--Kutta schemes is their implicit nature, which requires solving a nonlinear system at each step. By introducing a new set of order conditions which minimise the antisymmetric component of a scheme, we derive what we call Explicit and Effectively Symmetric (EES) schemes -- a new class of explicit Runge--Kutta schemes with near-symmetric properties. We present examples of second-order EES schemes and demonstrate that, despite their low order, these schemes readily outperform higher-order explicit schemes such as RK4 and RK5, and achieve results comparable to implicit symmetric schemes at a significantly lower computational cost.
\end{abstract}



\section{Introduction}
\label{sec:intro}

We consider a smooth vector field $f : \mathbb{R}^N \to \mathbb{R}^N$ and the corresponding first order autonomous ordinary differential equation (ODE)
\begin{equation}
\label{eq:ode}
	y'(t) = f(y(t)).
\end{equation} 
Given the initial condition $y(0) = p$, the corresponding flow map $\Phi_{s,f}: \mathbb{R}^N \to \mathbb{R}^N$, $s \in \mathbb{R}$, describes the solution $y(s) = \Phi_{s,f}(p)$ of \eqref{eq:ode} at $t=s$. As such, the flow map defines a 1-parameter family of diffeomorphisms preserving the composition  $\Phi_{s,f} \circ \Phi_{s',f} = \Phi_{s+s',f}$, and thus the inverse $\Phi_{-s,f} = \Phi_{s,f}^{-1}$. 

\smallskip

Finding numerical integrators, i.e., methods to compute numerical solutions of \eqref{eq:ode} plays a central role in many areas of applied mathematics and the mathematical sciences. A prominent class of numerical integration methods are Runge--Kutta methods, among others. Such methods are typically formulated as transformations: given the value of the solution at time $t$, the method applies a specific function to compute the value at time $t + h$, where $h$ denotes the step size. Thus, a numerical method with time step $h$ applied to the vector field $f$ is characterised by a map $\Psi_{h,f} : \mathbb{R}^N \to \mathbb{R}^N$ approximating the flow map. However, it is important to note a crucial difference between a numerical flow and the exact flow: $\Psi_{h,f}$ does not, in general, define a 1-parameter family, i.e., $\Psi_{h,f} \circ \Psi_{h',f}\neq \Psi_{h+h',f}$. Numerical methods that do, however, satisfy $\Psi_{-h,f} = \Psi_{h,f}^{-1}$ are called reversible or symmetric. More formally, such methods are characterised by invariance under the so-called symmetric-adjoint transformation, which sends $\Psi_h$ to $\Psi_{h,f}^* := \Psi^{-1}_{-h,f}$. 

Invariance and symmetry play fundamental roles in the formulation and analysis of numerical integrators for differential equations, especially when the goal is to retain the geometric and structural properties of the underlying continuous system \cite{Hairer-etal2006}. It is therefore crucial to identify particular properties of numerical methods that improve the preservation of such properties. These include symplecticity, energy conservation, and volume preservation -- each of which plays a critical role in the long-term qualitative behavior of the numerical solution. In this respect, symmetric integrators -- those invariant under time reversal -- are of special interest. They exhibit favorable long-term behavior and help preserve key qualitative features of the exact flow, which is essential for accurately capturing the system’s dynamics over extended time intervals. Classical applications of symmetric schemes include integrating physical systems. In particular, they are often well suited to the integration of Hamiltonian systems. In fact, in the case of linear Hamiltonian systems, the property of reversibility is equivalent to the symplecticity of a scheme \cite{feng2010symplectic}. For more details, we refer the reader to Chartier's concise presentation of symmetric methods \cite{Chartier2015}. In addition to classical applications, symmetric schemes have recently seen a significant interest in machine learning, where they have proven useful for training neural ODEs \cite{chen2018neural}. This has led to the design of several new symmetric schemes \cite{mccallum2024efficient, zhuang2021mali, kidger2021efficient}, typically of low order, to be used to train such models.\par\medskip

The main drawback of symmetric schemes lies in their typically low efficiency. Indeed, it is well known that any symmetric Runge–Kutta scheme must be implicit, leading to significantly higher computational costs compared to explicit methods. In addition, if the tolerance of the implicit root solver is set too large, one risks introducing substantial error into the solution and breaking the symmetry of the numerical scheme. Similarly to the case of Runge--Kutta schemes, it was shown in \cite{butcher2016symmetric} that symmetric parasitism-free general linear methods cannot be explicit.\par\medskip

Several approaches to the near-preservation of group symmetries have been proposed in the literature. A general recursive construction of integrators that are approximately invariant under a given symmetry is presented in \cite{iserles1999approximately}: given a desired symmetry $\Psi = \mathcal{A}\Psi$ where $\mathcal{A}$ is an automorphism, the proposed integrator is a composition of a chosen underlying integrator $\Phi$ with its \say{adjoint} $\mathcal{A}\Phi$, where the alternating pattern is governed by a Thue--Morse sequence. While this framework applies to a broad class of symmetries, the resulting integrators are trivially symmetric when $\mathcal{A}$ is an antiautomorphism, so the method cannot be applied to time-symmetry. Composition methods achieving near-preservation of time-symmetry have been discussed in \cite{casas2021compositions} in the context of double-jump compositions with complex-valued step sizes. There, the authors increase the order of existing symmetric methods via composition whilst maintaining approximate symmetry; following \cite{chartier1998reversible} and \cite{aubry1998pseudo}, they refer to such methods as \say{pseudo-symmetric}. Crucially, these prior approaches rely on composing existing symmetric schemes rather than constructing fundamentally new ones. In particular, when applied to time-symmetric Runge--Kutta methods, they cannot yield explicit schemes.\par\medskip

Many numerical integration methods on Euclidean spaces can be systematically studied through the formalism of B-series, introduced by John Butcher \cite{Hairer-etal2006,McLachlanetal2017,Butcher2021}. A B-series is an infinite expansion indexed by non-planar rooted trees, where each term is determined by a tree function $\psi$ and a vector field $f$. It encodes the action of a numerical integrator as a time-step map approximating the flow of a differential equation.

A key feature of B-series is that their composition yields another B-series. This closure under composition defines a natural group structure on the set of tree functions, reflecting the composition of the corresponding integration methods. Algebraically, this structure is dual to the coproduct of the Butcher--Connes--Kreimer Hopf algebra, a combinatorial Hopf algebra defined on forests of non-planar rooted trees \cite{connes1999hopf} defined in terms of admissible edge cuts on trees. In this framework, tree functions appearing in a B-series are precisely the Hopf algebra characters of $\mathcal{H}$.

Properties of numerical integrators can thus be studied in terms of rooted trees and the associated tree map $\psi$. Combinatorial arguments on trees determine the structure of the B-series required for the method to have specific structural properties. For instance, classical symmetry conditions on the Butcher tableau of a Runge--Kutta scheme correspond to constraints on so-called edge cuts of rooted trees, which therefore directly link to the coproduct in $\mathcal{H}$. Consequently, symmetry properties of general B-series methods can be formulated in terms of the underlying Hopf algebraic structure.

\smallskip

We show that every B-series method admits a canonical factorisation into the composition of two components, which we refer to as the symmetric and antisymmetric parts. This factorisation is rooted in the odd-even decomposition of Hopf algebra characters, as introduced by Aguiar, Bergeron, and Sottile in their study of combinatorial Hopf algebras \cite{aguiar2006combinatorial}. Specifically, any character $\psi$ on a graded connected Hopf algebra such as the Butcher--Connes--Kreimer Hopf algebra $\mathcal{H}$ can be uniquely decomposed as
$\psi = \psi^{\text{even}} * \psi^{\text{odd}}$, where $\psi^{\text{even}}$ and $\psi^{\text{odd}}$ are characters satisfying certain parity constraints under the grading of the Hopf algebra, and $*$ denotes the convolution (or Butcher) product. In the context of B-series, this decomposition translates into a composition of two B-series methods -- one symmetric and one antisymmetric -- whose individual properties can be studied and controlled separately. We note that the odd-even decomposition at the level of characters has a combinatorial analog at the level of rooted trees, which turns out to be computationally useful.  

\smallskip

This structure-theoretic insight motivates the design and analysis of near-symmetric numerical schemes, wherein the antisymmetric component $\psi^{\text{even}}$ is either identically trivial or sufficiently small to ensure that the integrator retains desirable properties such as time-reversibility. Importantly, the symmetric component $\psi^{\text{odd}}$ often inherits or approximates these properties even when the full method does not. 

As an example of near-symmetric schemes, we introduce the class of so-called Explicit and Effectively Symmetric (EES) Runge--Kutta schemes, which benefit from being explicit whilst keeping the antisymmetric component small. The derivation of order conditions for these EES schemes is a significant task, which requires the symmetric decomposition of hundreds of rooted trees. To aid in this computation, we introduce the Kauri Python package \cite{shmelev2025kauri} for the algebraic manipulation of rooted trees. 

\medskip

The paper is organised into eight sections. Section \ref{sec:HA} briefly reviews the Hopf algebra of non-planar rooted trees, commonly known as the Butcher--Connes--Kreimer Hopf algebra. Section \ref{sec:methods} provides a short overview of B-series methods, with a particular focus on Runge--Kutta schemes. Section \ref{sec:adjointmethod} explores the concept of adjoint schemes through the lens of rooted trees and characters of the Butcher--Connes--Kreimer Hopf algebra. Section \ref{sec:evenodd} presents the general odd-even decomposition of characters over a combinatorial Hopf algebra due to Aguiar, Bergeron, and Sottile, and sets the stage for our approach to finding approximately symmetric methods. Section \ref{sec:rationalpowers} discusses square roots, or more generally rational powers, of the identity map on rooted trees, and how these can be used to recursively compute the odd-even decomposition of characters over the Butcher--Connes--Kreimer Hopf algebra. Section \ref{sec:equivalence} introduces a new concept of equivalence of Runge--Kutta schemes based on their symmetric component, and discusses properties of the resulting equivalence classes of schemes. Section \ref{sec:explicitschemes} is central to this work. It introduces the novel concept of Explicit and Effectively Symmetric (EES) schemes by exploiting the odd-even decomposition of Runge--Kutta schemes at the level of Hopf algebra characters. This is achieved by imposing additional order conditions which minimise the antisymmetric component, resulting in schemes that are both explicit and nearly symmetric. We explore the quality of EES schemes by comparison with other schemes. Section \ref{sec:conclusions} closes this work with a conclusion outlining applications and interesting questions. Appendices A and B study stability aspects of EES schemes and the symmetric components of classical methods.

\section*{Acknowledgments}
NT acknowledges funding by the Deutsche Forschungsgemeinschaft (DFG, German Research Foundation) – CRC/TRR 388 "Rough Analysis, Stochastic Dynamics and Related Fields" – Project ID 516748464.

\section{The Hopf Algebra of Non-Planar Rooted Trees}
\label{sec:HA}

A non-planar rooted tree is defined to be a tree $\tau = (V,E,r)$ with vertex set $V$ and edge set $E$, together with a distinguished vertex $r \in V$ called the root. Two trees $\tau = (V, E, r)$ and $\tau' = (V', E', r')$ are identified if there exists a bijective map $f : V \to V'$ such that $(e_1, e_2) \in E \Leftrightarrow (f(e_1), f(e_2)) \in E'$. As such, the trees we consider are \textit{unlabelled}. We will denote by $\mathcal{T}$ the set of such unlabelled non-planar rooted trees, to be understood as equivalence classes of labelled trees, including the empty tree which is denoted $\emptyset$. Given trees $\tau_1, \ldots, \tau_k \in \mathcal{T}$, we will write $[\tau_1, \ldots, \tau_k]$ to mean the tree formed by joining the root vertices of $\tau_1, \ldots, \tau_k$ to a new root vertex,
\begin{equation}
\label{treebracket}
    [\tau_1, \ldots, \tau_k] = \commonRootTree.
\end{equation} 
Recall that non-planarity means that the order of the trees in $[\tau_1, \ldots, \tau_k]$ is irrelevant. For example, we do not distinguish between the trees
\begin{equation*}
   [\,\RS{n}, \RS{i} \,] = \RS{nlRi}
    \quad \text{and} \quad 
    [\,\RS{i}, \RS{n} \,] = \RS{nlrLi}.
\end{equation*}
When the expression contains repeated trees, we will write these in power notation, for example,
\begin{equation*}
    [\tau_1,\tau_1, \tau_2, \tau_3, \tau_3, \tau_3] = [\tau_1^2, \tau_2, \tau_3^3].
\end{equation*}
We will write $|\tau|$ to denote the number of nodes in the tree, i.e., the cardinality $|V|$. Permutations of the set of vertex labels $V$ naturally give rise to a group of symmetries. We will denote the order of this group by $\sigma(\tau)$. Alternatively, $\sigma(\tau)$ can be defined recursively by
\begin{align*}
    &\sigma(\emptyset) = 1, \quad \sigma(\RS{n}) = 1,\\
    &\sigma([\tau_1^{k_1}, \ldots, \tau_m^{k_m}]) = \prod_{i=1}^m k_i! \sigma(\tau_i)^{k_i}.
\end{align*}
Similarly, the tree factorial is defined recursively by
\begin{align*}
    &\emptyset! = 1, \quad \RS{n} ! = 1, \\
    &[\tau_1, \ldots, \tau_n]! = |[\tau_1, \ldots, \tau_n]| \cdot \tau_1 ! \cdots \tau_n !.
\end{align*}
We will refer to the commutative juxtaposition of trees as a forest. For example,
\begin{equation*}
    \RS{nlr} \hspace{2mm} \RS{n} \hspace{2mm} \RS{nliRi}
\end{equation*}
is a forest obtained by juxtaposing three trees. We note that in reference \cite{connes1999hopf}, another notation is used for the bracket operation \eqref{treebracket} mapping forests to trees, $B_+(\tau_1 \cdots \tau_k)=[\tau_1, \ldots, \tau_k]$. 

The free commutative $\mathbb{R}$-algebra generated by $\mathcal{T}$ will be denoted $\mathcal{H}$. We will adopt the Hopf algebra structure on $\mathcal{H}$ presented in the work by Connes and Kreimer \cite{connes1999hopf}, commonly referred to as the Butcher--Connes--Kreimer Hopf algebra of rooted trees, which is defined as follows. We refer the reader to \cite{CarPat2021} for a detailed introduction to (combinatorial) Hopf algebras and their applications.

\medskip
\begin{itemize}

    \item Multiplication $\mu : \mathcal{H} \otimes \mathcal{H} \to \mathcal{H}$ is defined as the commutative juxtaposition of two forests, extended linearly to $\mathcal{H}$.\medskip
    
    \item The multiplicative unit is defined to be the empty forest $\emptyset$.\medskip
    
    \item The counit map $\varepsilon: \mathcal{H} \to \mathbb{R}$ is defined by $\varepsilon(\emptyset) = 1$ and $\varepsilon(\tau) = 0$ for all non-empty trees $\tau \in \mathcal{H}$.\medskip
    
    \item Comultiplication $\Delta : \mathcal{H} \to \mathcal{H} \otimes \mathcal{H}$ is defined as follows. A so-called cut of a tree $\tau$ is either a subset of edges $c \subset E$, or a special cut called the \say{total cut}, to be defined shortly. The cut containing no edges will be referred to as the \say{empty cut}. For $c \subset E$, consider the forest $\tau^c = (V, E\setminus c)$ formed by removing the edges in $c$. Denote by $R_c(\tau)$ the tree of $\tau^c$ which contains the original root of $\tau$, and by $P_c(\tau)$ the forest formed by the juxtaposition of the remaining connected components. The \say{total cut} is defined such that $R_c(\tau) = \emptyset$ and $P_c(\tau) = \tau$. Following \cite{foissy2013introduction}, we say a cut is \say{admissible} if any oriented path through the tree meets at most one edge in the cut. We set $\text{Adm}^*(\tau)$ to be the set of admissible cuts, and $\text{Adm}(\tau) = \text{Adm}^*(\tau) \setminus \{\text{\say{empty cut}, \say{total cut}}\}$. Then for $\tau \in \mathcal{T}$ we define
    \begin{equation*}
        \Delta(\tau) := \sum_{c \in \text{Adm}^*(\tau)} P_c(\tau) \otimes R_c(\tau) = \tau \otimes \emptyset + \emptyset \otimes \tau + \sum_{c \in \text{Adm}(\tau)} P_c(\tau) \otimes R_c(\tau).
    \end{equation*}
    Comultiplication is then extended inductively to forests by defining $\Delta(\tau_1 \tau_2) = \Delta(\tau_1)\Delta(\tau_2)$ for all forests $\tau_1, \tau_2$, and extended linearly to $\mathcal{H}$. It will be convenient to additionally define the reduced coproduct
    \begin{equation*}
        \delta(\tau) := \sum_{c \in \text{Adm}(\tau)} P_c(\tau) \otimes R_c(\tau)
    \end{equation*}
    such that $\Delta(\tau) = \tau \otimes \emptyset + \emptyset \otimes \tau + \delta(\tau)$. For example,
    \begin{align*}
    \Delta\left(\RS{lr}\right) &= \RS{lr} \otimes \emptyset 
    + \emptyset \otimes \RS{lr} 
    + 2\, \RS{n} \otimes \RS{i} 
    + \RS{n}\, \RS{n} \otimes \RS{n},\\
    \delta\left(\RS{lr}\right) &= 2\, \RS{n} \otimes \RS{i} 
    + \RS{n}\, \RS{n} \otimes \RS{n}.
    \end{align*}
    
    \item There are several equivalent ways of defining the antipode $S$ of $\mathcal{H}$. For our purposes, we will define
    \begin{equation}
    \label{antipode}
        S(\tau) := \sum_{c \text{ a cut of } \tau} (-1)^{|c| + 1} \tau^c,
    \end{equation}
    known as the forest formula, which can be shown to be equivalent to the more common recursive definition (see for instance \cite[Theorem 2]{foissy2013introduction}). We will occasionally write $S\tau$ in place of $S(\tau)$ when the argument of $S$ is clear. For example,
    \begin{equation*}
        S\, \RS{n} = -\, \RS{n},
        \qquad
        S\, \RS{i} = -\, \RS{i} + \, \RS{n}\, \RS{n},
        \qquad
        S\,\RS{lr}=
        - \RS{lr} + 2\, \RS{n}\, \RS{i} 
        - \RS{n}\, \RS{n}\, \RS{n}\,.
    \end{equation*}
\end{itemize}

The resulting Hopf algebra $(\mathcal{H}, \Delta,\mu,\varepsilon, \emptyset, S)$ is commutative, non-cocommutative, graded by the total number of vertices in a forest, i.e., $\mathcal{H} = \bigoplus_{n \geq 0} \mathcal{H}_n $, and connected, $\mathcal{H}_0=\mathbb{R}\emptyset$. See \cite{hoffman2003hopf} for a comprehensive presentation.

\section{B-Series and Runge--Kutta Methods}
\label{sec:methods}

Throughout, we will consider a smooth vector field $f: \mathbb{R}^{{N}} \to \mathbb{R}^{{N}}$ and the corresponding ordinary differential equation (ODE) 
\begin{equation}
\label{ODE}
 \frac{dy}{dt} = f(y).   
\end{equation}
We briefly recall Butcher's B-series expansion associated with the ODE \eqref{ODE}. For any tree $\tau \in \mathcal{T}$, the so-called elementary differential $F(\tau)(y)$ \cite[Definition 310A]{butcher2016numerical} is defined recursively by
\begin{align*}
    &F(\emptyset)(y) = y, \quad F(\RS{n})(y) = f(y),\\
    &F([\tau_1, \tau_2, \ldots, \tau_m])(y) = f^{(m)}(y)(F(\tau_1)(y), F(\tau_2)(y), \ldots, F(\tau_m)(y)).
\end{align*}
For example, choosing coordinates to express the vector field $f=\sum_{j=1}^{{N}} f^j{\partial_j}$ on $\mathbb{R}^{{N}}$, we find 
\begin{equation*}
    F([\RS{n}\, \RS{n}])=F\left(\RS{lr}\right)=f^{(2)}(f,f)= \sum^N_{j,k,m=1} f^kf^m(\partial_k\partial_mf^j) \partial_j.
\end{equation*}

Given a map $\varphi : \mathcal{T} \to \mathbb{R}$, we define the associated B-series as 
\begin{equation*}
    B_h(\varphi, y_0) := \sum_{\tau \in \mathcal{T}}  \frac{h^{|\tau|}}{\sigma(\tau)} \varphi(\tau) F(\tau)(y_0).
\end{equation*}
The Butcher group $G$ is defined to be the set of algebra homomorphisms $\varphi : \mathcal{H} \to \mathbb{R}$, that is, $\varphi(\emptyset) = 1$ and $\varphi(\tau_1\tau_2) = \varphi(\tau_1)\varphi(\tau_2)$, with the group structure defined in terms of the convolution product
\begin{equation*}
    \varphi_1 * \varphi_2 = \mu_\mathbb{R} \circ (\varphi_1 \otimes \varphi_2) \circ \Delta,
\end{equation*}
with $\mu_\mathbb{R} : \mathbb{R} \otimes \mathbb{R} \to \mathbb{R}$ denoting multiplication in $\mathbb{R}$. For example,
\begin{equation*}
    \varphi_1 * \varphi_2\left(\RS{nlr}\right)
    = \varphi_1\left(\RS{nlr}\right) + \varphi_2\left(\RS{nlr}\right)
    + 2\, \varphi_1 (\RS{n}) \,\varphi_2\left(\,\RS{ni}\,\right)
    + \varphi_1 (\RS{n})\,\varphi_1 (\RS{n}) \,\varphi_2(\RS{n})\,.
\end{equation*}
Unless stated otherwise, this product will always be implied by writing $\varphi_1 \varphi_2$, with the power notation $\varphi^n = \varphi \varphi \cdots \varphi$ denoting repeated application of the product. The inverse in the Butcher group is given by composition with the antipode \eqref{antipode}, $\varphi^{-1} = \varphi \circ S$. The identity in $G$ is given by the counit $\varepsilon$. Equipped with this product, one can show that for two B-series \cite{hairer1974butcher, butcher2024b},
\begin{equation*}
    B_h(\varphi_2,B_h(\varphi_1,y_0)) = B_h(\varphi_1\varphi_2,y_0).
\end{equation*}

That is, the group product in $G$ captures exactly the composition of B-series. The B-series expansion of the exact solution to the ODE \eqref{ODE} is given by
\begin{align} 
\label{eq:exact_butcher_series}
    y(s) &= y_0 + \sum_{\tau \in \mathcal{T}\setminus\{\emptyset\}} \frac{s^{|\tau|}}{\sigma(\tau)} \frac{1}{\tau! } F(\tau)(y_0)\\
    &= B_s\left(a, y_0\right),
\end{align}
where the tree map is given in terms of the inverse tree factorial $a(\tau) = 1/\tau!$.

\begin{definition}
A B-series method is any formal power series $B_h(\varphi, f)$, where $\varphi \in G$. We call $\varphi$ the corresponding \say{elementary weights function} or simply the corresponding \say{character} of the B-series method.
\end{definition}

Of particular interest to us is a specific subclass of B-series methods: the Runge–Kutta methods. Consider solving the original ODE \eqref{ODE} numerically using a one-step method 
\begin{equation}
\label{1-stepmethod}
    y_{n+1} = y_n + h \Psi(y_n, h).
\end{equation}
Then \eqref{1-stepmethod} is an $s$-stage Runge--Kutta (RK) method if $\Psi$ can be written in the form
\begin{align*}
    y_{n+1} &= y_n + h \sum_{i=1}^s b_i k_i,\\
    k_i &= f\left(y_n + h\sum_{j=1}^{s} a_{ij} k_j\right), \quad i=1,\ldots,s,
\end{align*}
with RK matrix $A = (a_{ij})_{1 \leq i,j \leq s}$, weights $b = (b_i)_{1 \leq i \leq s}$ and nodes $c = (c_i)_{1 \leq i \leq s}$.
We will make the standard consistency assumption
\begin{equation}
    \sum_{i=1}^s b_i = 1 \label{eq:RK_assumption1}
\end{equation}
for all of the schemes we consider, and define $c_i := \sum_{j=1}^s a_{ij}$. 

\begin{definition}
\label{def:adjointmethod}
For the one-step method \eqref{1-stepmethod}, the \say{adjoint method} is defined by 
$$
    \Psi^*_h := \Psi^{-1}_{-h}.
$$
The scheme is said to be \say{symmetric}, or \say{reversible}, if $\Psi = \Psi^*$.
\end{definition}

\begin{theorem}[{\cite[Theorem 8.3]{hairer1993nonstiff}}] \label{thm:RK_symmetric}
Let $\Psi$ be an $s$-stage Runge--Kutta method. Then the adjoint method is also an $s$-stage Runge--Kutta method with coefficients
    \begin{align}
    \label{adjRKmethod}
    \begin{aligned}
    a_{ij}^* &= b_{s+1-j} - a_{s+1-i, s+1-j}\\
    b_j^* &= b_{s+1-j}.  
    \end{aligned}
    \end{align}
\end{theorem}

\begin{remark} 
\label{rmk:index_order}
The permutation of indices $i \mapsto s + 1 - i$ in \eqref{adjRKmethod} is purely to preserve the order of $c_1, \ldots, c_s$, and is not strictly necessary.
\end{remark}

There is a simple, canonical way in which one may construct a symmetric scheme, which is to compose an existing scheme with its adjoint.

\begin{proposition*}
\label{prop:composition}
Let $\Psi$ be a numerical scheme. Then the methods
    \begin{align*}
        \Psi'_h &= \Psi^*_{h/2} \circ \Psi_{h/2}\\
        \Psi''_h &= \Psi_{h/2} \circ \Psi^*_{h/2}
    \end{align*}
are symmetric.
\end{proposition*}

\begin{theorem}[{\cite[Theorem 8.8]{hairer1993nonstiff}}]
The Runge--Kutta method $\Psi$ is symmetric if
    \begin{equation*}
        a_{s+1-i, s+1-j} + a_{ij} = b_{s+1-j} = b_j, \quad i,j = 1,\ldots, s.
    \end{equation*}
Moreover, if the $b_i$ are non-zero and the $c_i$ distinct and ordered as $c_1 < c_2 < \cdots < c_s$, then the condition is also necessary.
\end{theorem}

\begin{definition}[{\cite[Definition 312A]{butcher2016numerical}}]
Consider an $s$-stage Runge--Kutta method. The \say{elementary weights} $\psi(\tau)$, the \say{internal weights} $\psi_i(\tau)$ and the \say{derivative weights} $(\psi_i D)(\tau)$ for any tree $\tau \in \mathcal{T}$ and $i = 1, \ldots, s$ are defined inductively by
    \begin{align*}
        (\psi_i D)(\RS{n}) &= 1, \\
        \psi_i(\tau) &= \sum_{j=1}^s a_{ij} (\psi_j D)(\tau),\\
        (\psi_i D)([\tau_1, \tau_2, \ldots, \tau_k]) &= \prod_{j=1}^k \psi_i(\tau_j),\\
        \psi(\tau) &= \sum_{i=1}^s b_i(\psi_i D)(\tau).
    \end{align*}
\end{definition}
    
We will extend these definitions to forests multiplicatively so that $\psi(\tau_1 \tau_2) = \psi(\tau_1)\psi(\tau_2)$.
    
\begin{lemma} [{\cite[Lemma 312B]{butcher2016numerical}}]
\label{lemma:elementary_weights}
    Let $\tau = (V,E)$ and $|\tau| = n$. Then
    \begin{align*}
        \psi(\tau) &= \sum_{i_1, \ldots, i_n} b_{i_1} \prod_{(k,\ell) \in E} a_{i_k, i_\ell}\\
        \psi_i(\tau) &=  \sum_{i_1, \ldots, i_n} a_{i, i_1} \prod_{(k,\ell) \in E} a_{i_k, i_\ell}.
    \end{align*}
\end{lemma}

The approximation $y_\Psi$ given by the RK scheme then admits the B-series expansion $y_\Psi = B_h(\psi, y_0)$.

\begin{notation}
Throughout, we will always denote RK schemes by upper-case Greek letters, and use the corresponding lower-case letter to denote their elementary weights function.
\end{notation}

\begin{definition}
A one-step method $y_{n+1} = y_n + h \Psi(y_n, h)$ is said to be of order $p$ if
    \begin{equation*}
        \lVert y(h) - y_1 \rVert \leq C h^p
    \end{equation*}
for some constant $C$ independent of $h$.
\end{definition}

\begin{theorem} [{\cite[Theorem 315A]{butcher2016numerical}}]
\label{thm:order_p_weights}
A Runge--Kutta method with elementary weights $\psi : \mathcal{T} \to \mathbb{R}$ has order $p$ if and only if
    \begin{equation*}
        \psi(\tau) = \frac{1}{\tau!}, \quad \forall \tau \in \mathcal{T} \text{ s.t. } |\tau| \leq p.
    \end{equation*}
\end{theorem}

\begin{definition}
We say that a B-series method $B_h(\varphi, f)$ is of order $p$ if $\varphi(\tau) = 1/\tau!$ for all $|\tau| \leq p$.
\end{definition}

\begin{theorem} [{\cite[Theorem 317A]{butcher2016numerical}}]\label{thm:rk_density}
Given a finite subset $\mathcal{T}_0$ of $\mathcal{T}$ and a mapping $f : \mathcal{T}_0 \to \mathbb{R}$, there exists a Runge--Kutta method such that the elementary weights satisfy
    \begin{equation*}
        \psi(\tau) = f(\tau), \quad \forall \tau \in \mathcal{T}_0.
    \end{equation*}
\end{theorem}

\begin{remark}\label{rmk:explicit_rk_density}
    We note that the same proof as the one presented in \cite{butcher2016numerical} shows that Theorem \ref{thm:rk_density} holds if one restricts the space of Runge--Kutta methods to only include implicit methods, or similarly only explicit methods. That is, given the map $f$ which we aim to match in Theorem \ref{thm:rk_density}, we are free to construct our Runge--Kutta method to be either implicit or explicit. This exposes two deeper facts about these two types of schemes. Firstly, it must be the case that explicit schemes are \say{dense} in Runge--Kutta schemes, in the sense that they can be used to approximate any Runge--Kutta scheme up to arbitrary order. This fact will underpin the construction of EES schemes in Section \ref{sec:explicitschemes}. Secondly, it is clear that one cannot verify whether a scheme is implicit or explicit knowing only finitely many values of $\psi(\tau)$, as both types of scheme may be constructed to match these values.
\end{remark}

\section{The Hopf Algebraic Structure of The Adjoint Method}
\label{sec:adjointmethod}

We are interested in studying the structure of symmetric schemes from a Hopf algebraic point of view. As such, a natural starting point is to understand the construction of the adjoint scheme at the level of trees and elementary weights, or characters. The following theorem is a simple corollary of the composition for B-series, but allows us to translate the idea of an adjoint method to that of an adjoint character.

\begin{definition}
The canonical involution on $\mathcal{H}$ is defined for $\tau \in \mathcal{H}$ by 
    \begin{equation}
    \label{invol}
        \bar\tau:=(-1)^{|\tau|}\tau.
    \end{equation}
\end{definition}

The involution \eqref{invol} induces an involution of characters over the Hopf algebra $\mathcal{H}$, given by 
$$
    \bar\zeta(\tau) := (-1)^{|\tau|}\zeta(\tau)
$$ 
for a character $\zeta$. It is easy to verify that the involution is multiplicative. That is, for any two characters $\zeta$ and $\xi$, we have $\overline{\zeta \xi} = \bar{\zeta} \bar{\xi}$ and $\overline{\zeta^{-1}} = (\bar{\zeta})^{-1}$.

\begin{theorem} 
\label{thm:psi_adj_weights}
Let $\psi, \psi^*$ be the elementary weight functions for the B-series methods $\Psi$ and $\Psi^*$ respectively, where $\Psi^*$ is the adjoint method of $\Psi$. Then
    \begin{equation}
    \label{eq:psi_adj}
        \psi^*(\tau) = (-1)^{|\tau|}\psi(S \tau).
    \end{equation}
\end{theorem}

\begin{proof}
Recall the involution $\bar\psi(\tau) := (-1)^{|\tau|}\psi(\tau)$. Then clearly $ y_{\Psi}(-h) = B_h(\bar\psi, y_0)$.
By definition, $\Psi^*_h \circ \Psi_{-h} = \text{Id}$, and so by the composition theorem for B-series
    \begin{equation*}
        B_h(\bar{\psi} \psi^*, y_0) = B_h(\psi^*, B_h(\bar{\psi}, y_0)) = B_h(\varepsilon, y_0).
    \end{equation*}
It follows that $\psi^* = \bar{\psi}^{-1} \varepsilon = \bar{\psi}^{-1}$.
\end{proof}

From now on, given any character $\zeta$, we define $\zeta^*$ to be the character $\zeta^*(\tau) = (-1)^{|\tau|} \zeta(S\tau)$, and we understand that this corresponds to the adjoint operation on the level of B-series methods.\par\medskip

While the above proof gives a clear account of the result in terms of B-series, it is unclear how this relates to Theorem \ref{thm:RK_symmetric} for RK schemes and the elementary weights given by Lemma \ref{lemma:elementary_weights}. In the remainder of this section, we will form the necessary link between these results. From now on, any equation labelled with $(\Sigma)$ will assume summation over all indices that appear, including those that would typically be considered “free” in the sense of Einstein summation convention. For example, given an RK scheme $\Psi$, we may write
\begin{equation*}
    \psi\left(\, \RS{ni}\, \right) = b_i c_i = b_i a_{ij}. \tag{\ensuremath{\Sigma}}
\end{equation*}
When dealing with elementary weights, if we wish to map from a scheme $\Psi$ to its adjoint $\Psi^*$, Theorem \ref{thm:RK_symmetric} leads us to the mapping sending $a_{ij} \mapsto b_{j} - a_{ij},$ where we have dropped the reordering of indices $i \mapsto s + 1 - i$ following Remark \ref{rmk:index_order}. This mapping has a rather natural interpretation when viewed as a mapping on trees. In an elementary weight, $a_{ij}$ represents an edge between $i$ and $j$. When we replace $a_{ij}$ by $b_j - a_{ij}$, we recover the sum of two elementary weights. The first is the weight corresponding to the monomial given by erasing the edge represented by $a_{ij}$ (but not the nodes $i$ and $j$, since these are reinstated as leaves or roots) and multiplying the resulting trees. The second is the negative of the original tree.
\begin{equation} \label{eq:adjoint_tree_map}
\raisebox{-1cm}{

    \treeFour \hspace{0.2cm}\raisebox{0.8cm}{$\mapsto$} \hspace{0.2cm} \treeFive \hspace{0.2cm}\raisebox{0.8cm}{$-$}\hspace{0.2cm} \treeFour
    
    }
\end{equation}

For example,
\begin{equation*}
    \psi\left(\, \RS{ni}\, \right) = b_i a_{ij} \mapsto b_i (b_j - a_{ij}) = b_i b_j - b_i a_{ij} =  \psi\left(\, \RS{n} \, \RS{n}\, \right) -  \psi\left(\, \RS{ni}\, \right). \tag{\ensuremath{\Sigma}}
\end{equation*}
\begin{alignat*}{2}
    \psi\left(\, \RS{nlr}\, \right) = b_i a_{ij} a_{ik} \xmapsto[]{\text{Map }a_{ij}}& \,\, b_i (b_j - a_{ij})a_{ik}\tag{\ensuremath{\Sigma}}\\
    =& \,\, b_i b_j a_{ik} - b_i a_{ij} a_{ik} &&= \psi(\, \RS{n} \, \RS{ni} \,) - \psi(\, \RS{nlr} \,)\\
    \xmapsto[]{\text{Map }a_{ik}}& \,\, b_i b_j (b_k - a_{ik}) - b_ia_{ij}(b_k - a_{ik})\\
    =& \,\, b_i b_j b_k - 2 b_i b_j a_{ik} + b_i a_{ij} a_{ik} &&= \psi(\,\RS{n} \, \RS{n} \, \RS{n}\,) - 2\psi(\,\,\RS{n}\, \RS{ni}\,) + \psi(\,\RS{nlr}\,).
\end{alignat*}
Viewing \eqref{eq:adjoint_tree_map} now as a mapping from trees to forests and dropping all references to $\psi$, we find for the first few trees in $\mathcal{T}$
\begin{align*}
    \RS{n} &\mapsto \RS{n}\\
    \RS{ni} &\mapsto \RS{n} \, \RS{n} - \RS{ni}\\
    \RS{nlr} &\mapsto \RS{n} \, \RS{n} \, \RS{n} - 2\,\RS{n}\, \RS{ni} + \RS{nlr}\\
    \RS{nIi} &\mapsto \RS{n} \, \RS{n} \, \RS{n} - 2 \, \RS{n} \, \RS{ni} + \RS{nIi}\\
    \RS{nlir} &\mapsto \RS{n} \, \RS{n} \, \RS{n} \, \RS{n}  - 3\,\RS{n}\,\RS{n}\,\RS{ni} + 3\,\RS{n}\,\RS{nlr} - \RS{nlir}\\
    \RS{nlRi} &\mapsto \RS{n} \, \RS{n} \, \RS{n} \, \RS{n} -3\,\RS{n}\,\RS{n}\,\RS{ni} + \RS{n}\,\RS{nlr} + \RS{ni}\,\RS{ni} + \RS{n}\,\RS{nIi} - \RS{nlRi}
\end{align*}
which, up to an overall sign, can be recognised as the antipode \eqref{antipode} on trees. This leads us to the second proof of Theorem \ref{thm:psi_adj_weights} for the case when $\Psi$ is an RK method.\par\medskip

\newenvironment{proof2}{\paragraph{$\textit{2}^{\,\textit{nd}}$ Proof of Theorem \ref{thm:psi_adj_weights} for RK methods}}{\hfill$\square$}

\begin{proof2}
    Recall that the antipode \eqref{antipode} on trees is defined as
    \begin{equation*}
        S\tau = \sum_{c \text{ a cut of } \tau} (-1)^{|c| + 1} \tau^c.
    \end{equation*}
    It is easy to see that this can be rewritten as the product of operators
    \begin{equation*}
        S\tau = -\left[\prod_{e \in E} (\text{Id} - \mathfrak{C}_e)\right](\tau),
    \end{equation*}
    where $\mathfrak{C}_e(\tau) := \tau^{\{e\}}$ is the operator on forests which cuts the edge $e$ in $\tau$. This is precisely how $\psi^*$ is constructed via the map $a_{ij} \mapsto b_j - a_{ij}$, up to a difference in sign which is corrected by a factor of $(-1)^{|\tau|}$.
\end{proof2}

\begin{corollary} \label{cor:symmetric_weights}
A B-series method is symmetric if and only if
    \begin{equation*}
        \psi(\tau) = (-1)^{|\tau|}\psi(S \tau) \quad \forall \tau \in \mathcal{T}.
    \end{equation*}
\end{corollary}

\begin{proof}
    \begin{align*}
        \Psi = \Psi^* &\Leftrightarrow \psi(\tau) = \psi^*(\tau) \quad \forall \tau \in \mathcal{T}\\
        &\Leftrightarrow \psi(\tau) = (-1)^{|\tau|}\psi(S \tau) \quad \forall \tau \in \mathcal{T}
    \end{align*}
\end{proof}

\begin{example} 
The implicit midpoint rule is given by the Butcher tableau
\[
\begin{array}{c|c}
c & A \\
\hline
 & b^T
\end{array}
=
\begin{array}{c|c}
1/2 & 1/2\\
\hline
& 1
\end{array}
\]

It is easy to see that this method is symmetric, and we can check that the result of Corollary \ref{cor:symmetric_weights} holds for the first few trees. Since the method has order $2$, by Theorem \ref{thm:order_p_weights} it must be the case that $\psi(\tau) = 1/\tau!$ for all $|\tau| \leq 2$. Since the result of Corollary \ref{cor:symmetric_weights} holds when $\psi(\tau)$ is replaced by the reciprocal of the tree factorial, it is only the cases where $|\tau| > 2$ which are of interest when checking the scheme. For brevity, we will remove redundant factors of $\RS{n}$ in the Table \ref{tab:adjoint_trees} below, since $\psi(\RS{n}) = \sum_i b_i = 1$.
\begin{table}[H]
    \centering
    {\renewcommand{\arraystretch}{1.7}
  \begin{tabular}{ccccc}
    $\tau$ & $1 / \tau!$ & $\psi(\tau)$ & $(-1)^{|\tau|}S\tau$ & $(-1)^{|\tau|}\psi(S\tau)$ \\
    \cline{1-5}
    \RS{n} & 1 & 1 & \RS{n} & 1 \\
    \RS{ni} & 1/2 & 1/2 & \RS{n} - \RS{ni} & 1/2 \\
    \RS{nlr} & 1/3 & 1/4 & \RS{n} - 2 \RS{ni} + \RS{nlr} & 1/4 \\
    \RS{nIi} & 1/6  & 1/4 & \RS{n} - 2 \RS{ni} + \RS{nIi} & 1/4 \\
    \RS{nlir} & 1/4 & 1/8 & \RS{n} - 3 \RS{ni} + 3 \RS{nlr} - \RS{nlir} & 1/8 \\
    \RS{nlRi} & 1/8 & 1/8 & \RS{n} - 3 \RS{ni} + \RS{nlr} + \RS{ni}\,\RS{ni} + \RS{nIi} - \RS{nlRi} & 1/8 \\
    \RS{nIlr} & 1/12 & 1/8 & \RS{n} - 3 \RS{ni} + \RS{nlr} + 2 \RS{nIi} - \RS{nIlr} & 1/8\\
    \RS{nIIi} & 1/24 & 1/8 & \RS{n} - 3 \RS{ni} + \RS{ni}\,\RS{ni} + 2 \RS{nIi} - \RS{nIIi} & 1/8
  \end{tabular}}
\caption{Verification of Theorem \ref{thm:psi_adj_weights} for the implicit midpoint rule.}
\label{tab:adjoint_trees}
\end{table}
\end{example}

\begin{example}
The $2$-stage Gauss collocation method of order $4$ is given by the Butcher tableau
\[
\begin{array}{c|cc}
\frac{1}{2} - \frac{\sqrt{3}}{6} & \frac{1}{4} & \frac{1}{4} - \frac{\sqrt{3}}{6}\\
\frac{1}{2} + \frac{\sqrt{3}}{6} & \frac{1}{4} + \frac{\sqrt{3}}{6} & \frac{1}{4}\\
\hline
& \frac{1}{2} & \frac{1}{2}
\end{array}
\]

Since this method is of order $4$, we have $\psi(\tau) = 1 / \tau!$ for all trees $\tau \in \mathcal{T}$ of size $|\tau| \leq 4$ by Theorem \ref{thm:order_p_weights}. The relations for $|\tau| = 5$ are also not of particular interest, since the terms involving $\psi(\tau)$ cancel. However, for a tree $\tau$ with $|\tau| = 6$, for instance, one finds
\begin{equation*}
    \psi \,\RS{nrIlir} = b_i a_{ij} c_j^3 c_i \tag{\ensuremath{\Sigma}} = \frac{7}{144}\quad \left(\neq \frac{1}{\tau!} = \frac{1}{24}\right).
\end{equation*}
To compute $(-1)^{|\tau|}\psi(S\tau)$ we will require
\begin{equation*}
    \psi \,\RS{nIlir} = b_i a_{ij} c_j^3 \tag{\ensuremath{\Sigma}} = \frac{1}{18}\quad \left(\neq \frac{1}{\tau!} = \frac{1}{20}\right),
\end{equation*}
\begin{equation*}
    \psi \,\RS{nrIlr} = b_i a_{ij} c_j^2 c_i \tag{\ensuremath{\Sigma}} = \frac{5}{72}\quad \left(\neq \frac{1}{\tau!} = \frac{1}{15}\right),
\end{equation*}
with the remaining instances of $\psi$ equaling reciprocals of tree factorials. We can now check that (using \RS{n}\,\!\!-reduced notation)
\begin{align*}
    &\psi^* \,\RS{nrIlir} = \psi S \,\RS{nrIlir}\\
    &= \psi \left[ \RS{n} - 5\,\RS{ni} + 3\,\RS{nIi} + 4\,\RS{nlr} + 3\,\RS{ni}\,\RS{ni} - 3\,\RS{nIlr} - \RS{nlir} - 3\,\RS{nrLi} - 3\,\RS{ni}\,\RS{nlr} + \RS{nIlir} + \RS{ni}\,\RS{nlir} + 3\,\RS{nrIlr} - \RS{nrIlir} \right]\\
    &= 1 - 5 \cdot \frac{1}{2} + 3 \cdot \frac{1}{6} + 4 \cdot \frac{1}{3} + 3 \cdot \frac{1}{4} - 3 \cdot \frac{1}{12} - \frac{1}{4} - 3 \cdot \frac{1}{8} - 3 \cdot \frac{1}{6} + \frac{1}{18} + \frac{1}{8} + 3\cdot \frac{5}{72} - \frac{7}{144}\\
    &= 7/144\\
    &= \psi\,\RS{nrIlir}.
\end{align*}
\end{example}

\section{Odd and Even Characters}
\label{sec:evenodd}

Having understood what it means for B-series methods and RK schemes to be symmetric at the level of elementary weights $\psi$, we can begin to interpret this at the corresponding level of the rooted tree Hopf algebra $\mathcal{H}$. Recall that $\psi : \mathcal{H} \to \mathbb{R}$ is a multiplicative linear functional, i.e., a Hopf algebra \textit{character} on $\mathcal{H}$. Corollary \ref{cor:symmetric_weights} says precisely that the scheme is symmetric if and only if $\psi$ defines a so-called \textit{odd} character in the sense of Bergeron, Sottile, and Aguiar \cite{aguiar2006combinatorial}. 

\begin{definition}
    Given a character $\zeta$ on the Hopf algebra $\mathcal{H}$, recall the involution $\bar\zeta(\tau) := (-1)^{|\tau|}\zeta(\tau)$. We say that $\zeta$ is \textit{even} if $\bar\zeta = \zeta$ and \textit{odd} if $\bar\zeta = \zeta^{-1} = \zeta \circ S$.
\end{definition}

Whilst an odd character defines a symmetric method, an even character defines a method $\Psi$ satisfying $y_\Psi(s) = y_\Psi(-s)$. We will refer to such methods as \say{antisymmetric}, motivated by Theorem \ref{thm:odd_even_characters} and Corollary \ref{cor:RK_decomp} below. Given any character $\zeta$ and recalling \eqref{eq:psi_adj}, a canonical way of constructing an odd character is given by $\zeta^* \zeta$.

We immediately note that, on the level of numerical schemes, the above result is equivalent to Proposition \ref{prop:composition} when $\zeta$ is the character of a scheme. In fact, we will see that all symmetric B-series methods must be of this form.\par\medskip

An important result in the study of odd and even characters is that any character can be decomposed into the product of an odd and an even character \cite{aguiar2006combinatorial}. Recall that any character can be written as a sum of its graded components, $\zeta=\varepsilon + \sum_{n>0} \zeta_n$, where $\zeta_n$ is zero on $\mathcal{H}_m$, if $m \neq n$.

\begin{theorem}[{\cite[Theorem 1.5]{aguiar2006combinatorial}}]\label{thm:odd_even_characters}
    Let $\mathcal{H}$ be a graded connected Hopf algebra over a field $k$ of characteristic different from 2. Every character $\zeta$ decomposes uniquely as the product of an even character, $\zeta^+$, and an odd character, $\zeta^-$,
    \begin{equation*}
        \zeta = \zeta^+ \zeta^-.
    \end{equation*}
We have that $\zeta^- = (\zeta^+)^{-1} \zeta$ and $\zeta^+$ can be constructed as
\begin{align}
\label{eq:even_component_construction}
\begin{aligned}
        (\zeta^+)_0 &= \varepsilon\\
        (\zeta^+)_n &= 2^{-1} \left\{\bar\zeta_n - (\zeta^{-1})_n - \sum_{\substack{i+j+k = n\\0\leq i,j,k < n}} (\zeta^+)_i  (\zeta^{-1})_j (\zeta^+)_k\right\}.
\end{aligned}
\end{align}
\end{theorem}

In addition, the decomposition can be written as the product of an odd character with an even character by writing $\zeta = (\zeta^+ \zeta^- (\zeta^+)^{-1}) \zeta^+$, since odd characters remain odd after conjugation with an even character. A useful relation that must be satisfied by this decomposition is 
$$
    (\zeta^-)^2 = \zeta^* \zeta,
$$
which is the basis for the second part of the following corollary.

\begin{corollary}[Symmetric Decomposition of B-Series Methods]
\label{cor:RK_decomp}
Every B-series method $\Psi$ can be written as the composition
    \begin{equation}
    \label{eq:Psi_decomp}
        \Psi = \Psi^- \circ \Psi^+
    \end{equation}
    of two B-series methods $\Psi^-$ and $\Psi^+$, where $\Psi^-$ is symmetric and $\Psi^+$ is antisymmetric. Moreover, we have
    \begin{equation*}
        \Psi^- \circ \Psi^- = \Psi \circ \Psi^*.
    \end{equation*}
\end{corollary}

\begin{figure}[h]
    \centering
    \includegraphics[width = 0.6\textwidth, trim={1.1cm 0.4cm 1.1cm 0.5cm},clip]{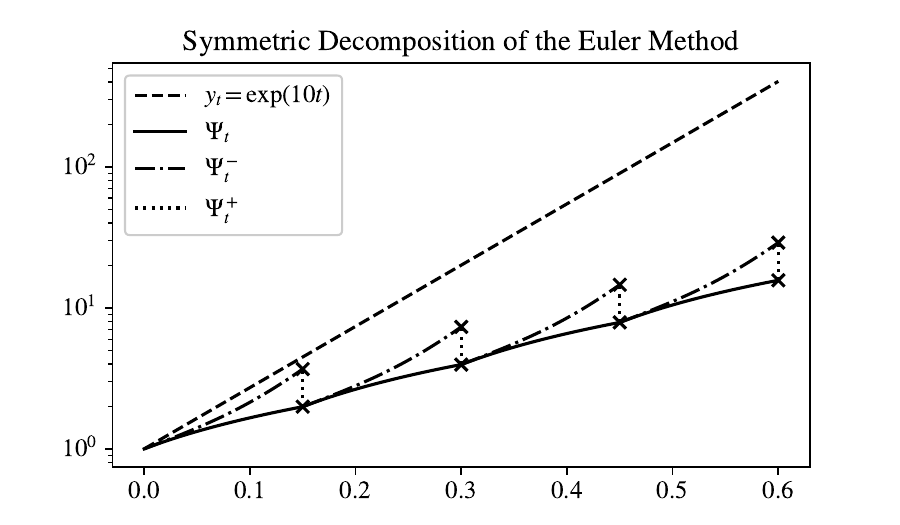}
    \caption{Decomposition of the Euler method $\Psi$ into its symmetric and antisymmetric components, evaluated on the ODE $dy/dt = 10y$, $y_0=1$ with a step size of $h=0.15$. The even component is viewed as a correction term connecting the odd component with the original scheme.}
\end{figure}

The structure resulting from the decomposition \eqref{eq:Psi_decomp} in Corollary \ref{cor:RK_decomp} can be viewed in different ways. For instance, one can view symmetric methods as a set of representatives for the left (and right) cosets of antisymmetric methods, when the latter is viewed as a subgroup of the Butcher group. Alternatively, it is clear that the relation $\Psi \sim \Phi \Leftrightarrow \Psi^- = \Phi^-$ is an equivalence relation on methods, and each equivalence class has a \say{central} element, given by the unique symmetric method in the class. The above corollary suggests that the canonical construction of symmetric methods given by Proposition \ref{prop:composition} is, in fact, the only possible form for a symmetric scheme. To establish this claim and analyze the properties of symmetric components, we choose to recover the symmetric part using the relation $(\psi^-)^2 = \psi^* \psi$, instead of relying on the construction given in \eqref{eq:even_component_construction}. As such, it will be prudent for us to briefly study the nature of square roots, or more generally rational powers, of characters on Hopf algebras.

\section{Rational Powers of Maps Over Hopf Algebras and \texorpdfstring{$\psi^-$}{ψ⁻}}
\label{sec:rationalpowers}

We will briefly generalise away from the Butcher--Connes--Kreimer Hopf algebra and prove a result about rational powers of linear multiplicative maps on a graded connected Hopf algebra $(\mathcal{H}, \Delta, \mu, \varepsilon, \nu, S)$. Naturally, such maps include characters via the inclusion $\nu(\mathbb{R}) = \mathcal{H}_0 \subset \mathcal{H}$. Throughout, we will write $\mathrm{Id}\colon \mathcal{H} \to \mathcal{H}$ to mean the identity map, $\operatorname{Id}(x) = x$, for all $x \in \mathcal{H}$.

\begin{definition}
Given two linear maps $f,g : \mathcal{H} \to \mathcal{H}$, we define the product map by convolution
    \begin{equation*}
        fg = \mu \circ (f \otimes g) \circ \Delta,
\end{equation*} 
and denote by power notation the $n$-fold product $f^n = ff\cdots f$.
\end{definition}
  It is a standard fact that the convolution product of two multiplicative maps is again multiplicative.
  Since \(\mathcal{H}\) is graded and connected, there are well-defined maps \(\exp\colon\operatorname{End}(\mathcal{H})\to\operatorname{End}(\mathcal{H})\) and \(\log\colon\operatorname{End}(\mathcal{H})\to\operatorname{End}(\mathcal{H})\).
  If \(\mathcal{H}\) is commutative, it is a standard fact that if \(\psi\in\operatorname{End}(H)\) is multiplicative then \(\alpha=\log\psi\) is infinitesimal, i.e., it satisfies \(\alpha(xy)=\varepsilon(x)\alpha(y)+\alpha(x)\varepsilon(y)\).
  Conversely, if \(\alpha\) is infinitesimal, its exponential \(\psi=\exp\alpha\) is multiplicative \cite{CarPat2021}.

  We now use this correspondence to define rational powers of multiplicative maps.
  \begin{theorem}\label{thm:hopf_roots}
    Let \(\mathcal{H}\) be a graded connected and commutative Hopf algebra, \(\psi\in\operatorname{End}(\mathcal{H})\) be multiplicative.
    For every \(q=\tfrac{m}{n}\in\mathbb{Q}\) there exists a unique multiplicative map \(\psi^q\in\operatorname{End}(\mathcal{H})\) such that \((\psi^q)^n=\psi^m\).
    It is characterised by the recursion
    \begin{equation}\label{eq:square_root_recursion}
      \psi^q = \frac1n\left[ \psi^m - \sum_{k=2}^n\binom{n}{k}\mu^{(k-1)}\circ(\psi^q)^{\otimes k}\circ\widetilde{\Delta}^{(k-1)} \right].
    \end{equation}
  \end{theorem}
  \begin{proof}
    It suffices to define \(\psi^q = \exp(q\log\psi)\), since then
    \[
      (\psi^q)^n = \exp(nq\log\psi) = \exp(m\log\psi) = \psi^m.
    \]
    Let us show that \(\psi^q=\psi\circ\mathrm{Id}^q\), where
    \[
      \mathrm{Id}^q \coloneqq \exp\left( q\log(\mathrm{Id}) \right)=\sum_{k=0}^\infty\frac{q^k}{k!}\mathfrak{e}^k
    \]
    and \(\mathfrak{e}\coloneqq\log(\mathrm{Id})\in\operatorname{End}(\mathcal{H})\) is the Eulerian idempotent, given by
    \[
      \mathfrak{e} = \sum_{k=1}^\infty \frac{(-1)^{k-1}}{k}J^k = \sum_{k=1}^\infty\mu^{(k-1)}\circ\widetilde{\Delta}^{(k-1)},
    \]
    where \(J\coloneqq\mathrm{Id}-\nu\circ\varepsilon\) is the canonical projection of \(\mathcal{H}\) onto \(\mathcal{H}_+=\ker\varepsilon\), and \(\mu^{(k-1)}\colon\mathcal{H}^{\otimes k}\to\mathcal{H}\) and \(\widetilde{\Delta}^{(k-1)}\colon\mathcal{H}\to\mathcal{H}^{\otimes k}\) are the iterated versions of the product and reduced coproduct, respectively.

    Indeed, note that since \(\psi\) is multiplicative we have
    \begin{align*}
      \log\psi &= \sum_{k=1}^{\infty}\frac{(-1)^{k-1}}{k}\mu^{(k-1)}\circ(\psi-\nu\circ\varepsilon)^{\otimes k}\circ\Delta^{(k-1)}\\
               &= \sum_{k=1}^{\infty}\frac{(-1)^{k-1}}{k}\mu^{(k-1)}\circ\psi^{\otimes k}\circ\widetilde{\Delta}^{(k-1)} \\
               &= \psi\circ\left( \sum_{k=1}^\infty \frac{(-1)^{k-1}}{k}\mu^{(k-1)}\circ\widetilde{\Delta}^{(k-1)} \right)\\
               &= \psi\circ\mathfrak{e}.
    \end{align*}
    For the same reason we also have \((\log\psi)^k=\psi\circ\mathfrak{e}^k\) for all \(k\ge 1\), from where the claim follows.

    To show the recursion, note that the identity \((\psi^q)^n=\psi^m\) implies that \(\psi^q\circ\mathrm{Id}^n=\psi^m\). Therefore
    \[
      \psi^q\circ\left( \mathrm{Id}^n-n\mathrm{Id} \right) + n\psi^q = \psi^m
    \]
    which implies the result, after noting that on \(\mathcal{H}_+\)
    \begin{align*}
      \mathrm{Id}^n - n\mathrm{Id} &= (J+u\circ\varepsilon)^n - n\mathrm{Id}\\
                                   &= \sum_{k=2}^n\binom{n}{k}J^k\\
                                   &= \sum_{k=2}^n\binom{n}{k}\mu^{(k-1)}\circ\widetilde{\Delta}^{(k-1)}.
    \end{align*}
  \end{proof}

\begin{remark}
Following \cite{CarPat2021}, we note that the $k$-th convolution power of the $\mathrm{Id}$ map is known as the $k$-th dilation (or $k$-th Adams operation). These maps are instrumental in the context of classical structure theorems for graded Hopf algebras. It is easy to see that $\mathrm{Id}^k \mathrm{Id}^l=\mathrm{Id}^{k+l}$, $\mathrm{Id}^k \circ \mathrm{Id}^l=\mathrm{Id}^{kl}$, and $(\mathrm{Id}^k \mathrm{Id}^l)\circ \mathrm{Id}^m=(\mathrm{Id}^k\circ \mathrm{Id}^m) (\mathrm{Id}^l\circ \mathrm{Id}^m)$, which are the underpinning properties of the algebra.
\end{remark}

\begin{proposition} \label{prop:odd_roots}
    Let $\mathcal{H}$ be a Hopf algebra satisfying the conditions of Theorem \ref{thm:hopf_roots}. If $\psi$ is an odd character over $\mathcal{H}$, then so is $\psi^q$ for all $q \in \mathbb{Q}$.
\end{proposition}

\begin{proof}
    Let $q = m/n$. Then $\phi = \psi^q$ is the unique character satisfying $\phi^n = \psi^m$. Then since $\psi$ is odd,
    \begin{align*}
        \psi^{-m} &= (\phi^{-1})^n\\
        \psi^{-m} &= \bar{\psi}^m = \bar{\phi}^n
    \end{align*}
    so $\phi^{-1}$, $\bar{\phi}$ are both $n^{th}$ roots of $\psi^{-m}$. Thus $\bar{\phi} = \phi^{-1}$ by uniqueness of the $n^{th}$ root.
\end{proof}

We now apply our results to the Butcher--Connes--Kreimer Hopf algebra. Having derived an explicit expression for the rational power of a character, we now present a closed formula for the method $\psi^-$ in terms of $\psi$, as stated in the following corollary.

\begin{corollary} 
\label{cor:tree_minus}
    Let ${(\cdot)}^-:\mathcal{H} \to \mathcal{H}$ and ${(\cdot)}^+:\mathcal{H} \to \mathcal{H}$ be the maps defined for any $\tau \in \mathcal{H}$ by
    \begin{align*}
        \tau^- &= \mu\circ(\bar{S}\otimes \mathrm{Id})\circ\Delta \circ \mathrm{Id}^{1/2}(\tau),\\
        \tau^+ &= \mu\circ(\mathrm{Id} \otimes (\cdot)^- \circ S)\circ\Delta (\tau),
    \end{align*}
    where $\bar{S}(\tau) = (-1)^{|\tau|} S(\tau)$. Then for any B-series method $\Psi$,
    \begin{align*}
        \psi^-(\tau) &= \psi(\tau^-),\\
        \psi^+(\tau) &= \psi(\tau^+).
    \end{align*}
\end{corollary}
\begin{proof}
    The result follows immediately from the identities $(\psi^-)^2 = \psi^* \psi$ and $\psi^+ = \psi (\psi^-)^{-1}$, and Theorem \ref{thm:hopf_roots}.
\end{proof}

\begin{remark}
Corollary \ref{cor:tree_minus} shows that the odd-even decomposition of characters can be understood entirely at the level of trees. Indeed, it is easy to see that the maps $(\cdot)^-$ and $(\cdot)^+$ induce an \say{odd-even}-type decomposition at the level of rooted trees, since
\begin{equation*}
     (\cdot)^+ (\cdot)^- = \mathrm{Id}.
\end{equation*}
\end{remark}

\begin{remark}
    In practice, it is simpler to solve the recurrence \eqref{eq:square_root_recursion}, given in this case by
    \begin{align*}
        \phi(\emptyset) &= \emptyset,\\
        \phi(\tau) &= \frac{1}{2}\left[\mu \circ(\bar{S} \otimes \text{Id}) \circ\Delta(\tau) - \phi\circ\mu\circ\delta(\tau)\right]
    \end{align*}
    and set $\tau^- = \phi(\tau)$.
\end{remark}

\begin{table}[H]
    \centering
    {\renewcommand{\arraystretch}{1.7}
  \begin{tabular}{c|c|c|c}
    $\tau$ & $\mathrm{Id}^{1/2}(\tau)$ & $\tau^-$ & $\tau^+$\\
    \cline{1-4}
    \RS{n} & $\frac{1}{2}\,\RS{n}$ & $\RS{n}$ & $0$\\
    \RS{ni} & $\frac{1}{2}\,\RS{ni} - \frac{1}{8}\,\RS{n}$ & $\frac{1}{2}\,\RS{n}$ & $\RS{ni} - \frac{1}{2}\,\RS{n}$\\
    \RS{nlr} & $\frac{1}{2}\,\RS{nlr} - \frac{1}{4}\,\RS{ni}$ & $\RS{nlr}$ & $0$ \\
    \RS{nIi} & $\frac{1}{2}\,\RS{nIi} - \frac{1}{4}\,\RS{ni} + \frac{1}{16}\,\RS{n}$ & $\RS{nIi} - \RS{ni} + \frac{1}{2}\,\RS{n}$ & $0$\\
    \RS{nlir} & $\frac{1}{2}\,\RS{nlir} - \frac{3}{8}\,\RS{nlr} + \frac{1}{64}\,\RS{n}$ & $\frac{3}{2}\,\RS{nlr} - \frac{1}{4}\,\RS{n}$ & $\RS{nlir} - \frac{3}{2}\,\RS{nlr} + \frac{1}{4}\,\RS{n}$\\
    \RS{nlRi} & $\frac{1}{2}\,\RS{nlRi} - \frac{1}{8}\,\RS{ni}\,\RS{ni} - \frac{1}{8}\,\RS{nIi} - \frac{1}{8}\,\RS{nlr} + \frac{1}{16}\,\RS{ni} + \frac{1}{128}\,\RS{n} $ & $\frac{1}{2}\,\RS{nIi} + \frac{1}{2}\,\RS{nlr} - \frac{1}{2}\,\RS{ni} + \frac{1}{8}\,\RS{n}$ & $\RS{nlRi} - \frac{1}{2}\,\RS{nIi} - \frac{1}{2}\,\RS{nlr} + \frac{1}{8}\,\RS{n}$\\
    \RS{nIlr} & $\frac{1}{2}\,\RS{nIlr} - \frac{1}{4}\,\RS{nIi} - \frac{1}{8}\,\RS{nlr} + \frac{1}{8}\,\RS{ni} - \frac{1}{64}\,\RS{n}$ & $\RS{nIi} + \frac{1}{2}\,\RS{nlr} - \RS{ni} + \frac{1}{4}\,\RS{n}$ & $\RS{nIlr} - \RS{nIi} - \frac{1}{2}\,\RS{nlr} + \RS{ni} - \frac{1}{4}\,\RS{n}$\\
    \RS{nIIi} & $\frac{1}{2}\,\RS{nIIi} - \frac{1}{8}\,\RS{ni}\,\RS{ni} - \frac{1}{4}\,\RS{nIi} + \frac{3}{16}\,\RS{ni} - \frac{5}{128}\,\RS{n}$ & $\RS{nIi} - \RS{ni} + \frac{3}{8}\,\RS{n}$ & $\RS{nIIi} - \RS{nIi} + \frac{1}{2}\,\RS{ni} - \frac{1}{8}\,\RS{n}$
  \end{tabular}}
  \caption{$\bullet$\,-\,reduced values of $\mathrm{Id}^{1/2}(\tau)$, $\tau^-$ and $\tau^+$ for $|\tau| \leq 4$.}
\end{table}

\begin{proposition}\label{prop:plus_zero}
    $\tau^+ = 0$ for all $\tau$ with $|\tau|$ odd.
\end{proposition}

\begin{proof}
    Let $\psi$ be any character on $\mathcal{H}$. Since $\psi^+$ is an even character, we have $\overline{\psi^+} = \psi^+$. Thus, for $\tau$ with $|\tau|$ odd, $-\psi^+(\tau) = \psi^+(\tau)$, which implies $-\tau^+ = \tau^+ = 0$.
\end{proof}

A simple consequence of the square root construction is that any symmetric scheme can be written in the canonical form given by Proposition \ref{prop:composition}, as shown by the following theorem.

\begin{theorem} \label{thm:main}
    A B-series method $\Psi$ is symmetric if and only if there exists a B-series method $\Omega$ such that $\Psi = \Omega \circ \Omega^*$.
\end{theorem}

\begin{proof}
    If $\Psi$ is a symmetric B-series method, then $\psi$ is an odd character on $\mathcal{H}$. By Theorem \ref{thm:hopf_roots}, $\psi$ has a unique square root $\psi^{1/2}$, and by Proposition \ref{prop:odd_roots}, $\psi^{1/2}$ is an odd character. Then, by definition, $\Psi = \Omega \circ \Omega^*$ where $\Omega = \Psi^{1/2}$ is the symmetric B-series method corresponding to the odd character $\psi^{1/2}$. Conversely, if there exists a method $\Omega$ such that $\Psi = \Omega \circ \Omega^*$, then clearly $\Psi$ is symmetric.
\end{proof}

\section{An Equivalence Structure}
\label{sec:equivalence}

Several notions of equivalence between RK schemes have been introduced in the literature. The most common are equivalence of numerical solutions, equivalence of elementary weight functions and so-called \say{P-equivalence} concerning different parameterisations of a given scheme. All three of these notions can be shown to be equivalent to each other \cite[Theorem 381H]{butcher2016numerical}. Since the notions are identical, it will be convenient for us to define two RK schemes $\Psi, \Phi$ as \say{equivalent} if their elementary weight functions match, that is $\psi = \phi$.\par\medskip

In addition to these three notions, we introduce a new concept of equivalence concerning the symmetric component of a scheme, which we will refer to as S-equivalence. We define this notion as an equivalence on characters over $\mathcal{H}$, and extend it naturally to B-series methods.

\begin{definition}
Given two characters $\zeta, \xi$ on $\mathcal{H}$, we say they are S-equivalent if $\zeta^- = \xi^-$. Similarly, given two B-series methods $\Psi, \Phi$, we say they are S-equivalent if their characters are S-equivalent. The corresponding equivalence classes of characters are denoted by $[\cdot]_S$.
\end{definition}

Since $\psi^*\psi = (\psi^-)^2$ for any character $\psi$, the resulting equivalence relation is precisely the kernel of the function mapping $\psi \mapsto \psi^*\psi$ on characters on $\mathcal{H}$. Several natural questions arise when considering S-equivalence. Given the various existing notions of equivalence for Runge--Kutta schemes, one may ask whether S-equivalence, when restricted to this class, defines a genuinely distinct relation. In particular, whether two distinct Runge–-Kutta methods can be S-equivalent. Indeed, we will show that if an S-equivalence class contains a (consistent) RK character, then it contains infinitely many.

\begin{definition}
    For a character $\phi$ on $\mathcal{H}$, let $n(\phi)$ be the number of consistent RK characters in $[\phi]_S$.
\end{definition}

\begin{theorem}
The only values $n(\cdot)$ can take are $0$ and $\infty$, and both of these values are attained.
\end{theorem}

\begin{proof}
    Let $\phi$ be an odd character and suppose $[\phi]_S$ contains a consistent RK character $\psi$. For $\lambda \in \mathbb{R}_{>0}$, consider the RK scheme $\Omega_\lambda$ with Butcher tableau
    \begin{equation*}
        \begin{array}{c|cc}
        2\lambda & \lambda & \lambda\\
        -2\lambda & -\lambda & -\lambda\\
        \hline
        & \lambda & -\lambda
        \end{array}
    \end{equation*}
    and denote the corresponding character by $\omega_\lambda$. It is easy to see that $\Omega_\lambda$ is an antisymmetric RK scheme and is not equivalent to the identity scheme. Define the modified RK scheme
    \begin{equation*}
        \Psi_\lambda = \Psi \circ \Omega_\lambda.
    \end{equation*}
    Clearly, the schemes $\Psi_\lambda$ are distinct for distinct choices of $\lambda$, and are consistent. Moreover, the character of $\Psi_\lambda$ is given by
    \begin{equation*}
        \omega_\lambda \psi = (\omega_\lambda \psi^+) \phi \in [\phi]_S, \quad \forall \lambda \in \mathbb{R}_{>0}
    \end{equation*}
    since $\omega_\lambda$ is an odd character. Thus if $n(\phi) > 0$, then $n(\phi) = \infty$. It remains to show that $0$ and $\infty$ are attained by $n(\cdot)$. Clearly $\infty$ is attained by taking the equivalence class of any RK scheme. For the case $n(\cdot) = 0$, one may note that $[\varepsilon]_S$ is precisely the set of all even characters, and so cannot contain a consistent RK character. Alternatively, let $\phi$ be any odd character that is not an RK character. By Proposition \ref{prop:odd_roots}, $\phi^{1/2}$ is also an odd character. Suppose that there is a RK character $\psi \in [\phi^{1/2}]_S$. Then it must be the case that $\psi^* \psi = \phi^{1/2} \phi^{1/2} = \phi$. But $\psi^* \psi$ is a RK character, whilst $\phi$ is not, which is a contradiction. Thus $n(\phi^{1/2}) = 0$. 
\end{proof}

Another natural question is that of the order of methods within an S-equivalence class. As we will see, the order of a method is bounded by the order of its symmetric component, and so the order of any method in an S-equivalence class is at most the order of the odd element of that class.

\begin{definition}
    Given a character $\psi$, let
    \begin{align*}
        \ord(\psi) &= \sup\{n \in \mathbb{N} : \psi(\tau) = 1 / \tau! \quad \forall |\tau| \leq n\}\\
        \ord([\psi]_S) &= \sup_{\zeta \in [\psi]_S} \ord(\zeta)
    \end{align*}
\end{definition}

\begin{proposition}\label{prop:ord} Given a character $\psi$ and $q \in \mathbb{Q}$, let $\psi_q$ be the character given by $\psi_q(\tau) = q^{-|\tau|} \psi^q(\tau)$. Then:
\begin{enumerate}
    \item $\ord(\psi) = \ord(\psi^*)$
    \item $\ord(\psi_q) = \ord(\psi)$ for all $q \in \mathbb{Q}$, $q > 0$
    \item $\ord((\psi_1 \psi_2 \cdots \psi_n)^{1/n}) \geq \min\limits_{i=1,\ldots,n}(\ord(\psi_i))$
\end{enumerate}
\end{proposition}

\begin{proof} Let $a(\tau) = 1/\tau!$.\par\medskip
    \begin{enumerate}
        \item Since $a$ is the elementary weights function associated with the true solution of the ODE, it is clear that $a = a^*$. Let $\ord(\psi) = p$. Since $S\tau$ depends only on trees $\tau'$ with $|\tau'| \leq |\tau|$, it follows that for any $\tau$ such that $|\tau| \leq p$, we have
        \begin{equation*}
            1/\tau! = \psi(\tau) = a(\tau) = a^*(\tau) = \psi^*(\tau)
        \end{equation*}
        and so $\ord(\psi^*) \geq p = \ord(\psi)$. A similar argument shows that $\ord(\psi) \geq \ord(\psi^*)$.\par\medskip
        
        \item We first show the case $q = n \in \mathbb{Z}^+$. It is clear that, since $a$ is the elementary weights function of the exact solution of the ODE, we have
        \begin{equation*}
            B_h(a_n, f) = B_{n^{-1}h}(a^n, f) = B_h(a, f),
        \end{equation*}
        and so $a_n = a$. Thus if $\ord(\psi) = p$, we must have
        \begin{equation*}
            \psi_n(\tau) = a_n(\tau) = a(\tau) = \psi(\tau)
        \end{equation*}
        for all $|\tau| \leq p$. Thus, it is clear that we must have exactly $\ord(\psi_n) = p$.\par\medskip
        
        For $q \in \mathbb{Q}$, let $c_q$ be the map $c_q(\tau) = q^{-|\tau|} \tau$, such that $\psi_q = \psi^q \circ c_q$. It is easy to see from the definition of the coproduct $\Delta$ that for any character $\zeta$ and any $n \in \mathbb{Z}^+$,
        \begin{equation*}
            (\zeta \circ c_q)^n = \zeta^n \circ c_q.
        \end{equation*}
        It follows that
        \begin{align*}
            \ord(\psi_q) &= \ord((\psi_q)_n)\\
            &= \ord((\psi^q \circ c_q)^n \circ c_n)\\
            &= \ord((\psi^q)^n \circ c_q \circ c_n)\\
            &= \ord(\psi^m \circ c_m)\\
            &= \ord(\psi_m)\\
            &= \ord(\psi).
        \end{align*}\par\medskip
        
        \item Let $\min(\ord(\psi_i)) = p$. Recall from Theorem \ref{thm:hopf_roots} that for a character $\zeta$,  $\zeta^{1/n}(\tau)$ depends only on $\zeta(\tau')$ for $|\tau'| \leq |\tau|$, and the $n^{th}$ root is unique. It follows that for any $|\tau| \leq p$,
        \begin{align*}
            (\psi_1 \psi_2 \cdots \psi_n)^{1/n}(\tau) &= (a^n)^{1/n}(\tau)\\
            &= a(\tau)\\
            &= 1/\tau!
        \end{align*}
        and so $\ord((\psi_1 \psi_2 \cdots \psi_n)^{1/n}) \geq p$.
    \end{enumerate}
\end{proof}

\begin{proposition} \label{prop:equiv_class_order}
    For an odd character $\phi$, $\ord(\phi) = \ord([\phi]_S)$. That is, the order of any element of an $S$-equivalence class is at most the order of the symmetric element of that class.
\end{proposition}

\begin{proof}
    If $\ord(\phi) = \infty$, clearly $\ord([\phi]_S) = \infty$. For the finite case, let $\ord(\phi) = p$ and suppose that $\exists \psi \in [\phi]_S$ with $\ord(\psi) \geq p+1$. Then by Proposition \ref{prop:ord},
    \begin{align*}
        \ord(\phi) &= \ord((\psi^* \psi)^{1/2})\\
        &\geq \min(\ord(\psi^*), \ord(\psi))\\
        &= \ord(\psi)\\
        &\geq p+1,
    \end{align*}
    contradicting that $\ord(\phi) = p$. Thus $\ord([\phi]_S) \leq \ord(\phi)$. Since $\phi \in [\phi]_S$, clearly also $\ord([\phi]_S) \geq \ord(\phi)$.
\end{proof}

\section{Explicit and Effectively Symmetric Schemes}
\label{sec:explicitschemes}

Having developed the theory surrounding the decomposition of schemes into symmetric and antisymmetric components, we may now turn our attention to the problem of constructing efficient symmetric Runge--Kutta schemes. As alluded to in the introduction, the classical bottleneck of such schemes is their implicit nature, requiring an iterative root-solving algorithm to be run at each step of the solver. Whilst it is not possible to construct a Runge--Kutta scheme which is both symmetric and explicit, we may now leverage our symmetric decomposition (Corollary \ref{cor:RK_decomp}) to construct explicit schemes that are \emph{almost} symmetric, by suitably minimising their antisymmetric component. In many practical applications, this near-symmetry is expected to suffice for retaining the desirable properties of time-reversibility, while significantly reducing computational cost.

\begin{definition}
\label{def:antisym-order}
A B-series method $\Psi = B_h(\psi,f)$ is said to be of antisymmetric order $p$ if $$\sup\{n \in \mathbb{N} : \psi(\tau^+) = 0 \quad \forall |\tau| \leq n\} = p.$$ We write $\ord^+(\Psi) = \ord^+(\psi) = p$.
\end{definition}

\begin{definition}
\label{def:EESscheme}
A Runge--Kutta scheme $\Psi$ is an Explicit and Effectively Symmetric (EES) scheme of order $n$ and antisymmetric order $m$ if $\Psi$ is an explicit scheme satisfying both $\ord(\psi) = n$ and $\ord^+(\psi) = m$. Let $\mathrm{EES}(n,m)$ denote the class of such schemes.
\end{definition}

Recall that if a scheme is symmetric, i.e. $\ord^+(\psi) = \infty$, it must be of even order, $\ord(\psi) = 2k$ for some $k \in \mathbb{N}$ \cite[Theorem 8.10]{hairer1993nonstiff}. Similarly, it is clear from Proposition \ref{prop:plus_zero} that if $\ord^+(\psi) \geq \ord(\psi)$, then the scheme is of even order and odd antisymmetric order. That is, for $m \geq n$, the class $\mathrm{EES}(n,m)$ is non-empty only if $n$ is even and $m$ is odd. Indeed, the converse is also true.

\begin{theorem}[Existence of EES Runge--Kutta Schemes]
    Given any even integer $n > 0$ and any odd integer $m > n$, there exists a Runge--Kutta scheme belonging to $\mathrm{EES}(n,m)$.
\end{theorem}

\begin{proof}
    Let $\phi$ be any odd character with $\ord(\phi) = n$ and let $\zeta$ be any even character with $\ord^+(\zeta) = m$. Let $\gamma = \zeta\phi$. By definition, $\ord^+(\gamma) = \ord^+(\zeta) = m$. Since $\ord^+(\zeta) > \ord(\phi) = n$, we have $\ord(\gamma) \geq n$, and it follows from Proposition \ref{prop:equiv_class_order} that $\ord(\gamma) = n$. Moreover, it is clear that any other character $\psi$ satisfying $\psi(\tau) = \gamma(\tau)$ for all $|\tau| \leq m+1$ has $\ord(\psi) = n$ and $\ord^+(\psi) = m$. It follows from Theorem \ref{thm:rk_density} and Remark \ref{rmk:explicit_rk_density} that there must exist an explicit Runge--Kutta scheme $\Psi$ whose character $\psi$ satisfies $\psi(\tau) = \gamma(\tau)$ for all $|\tau| \leq m+1$, and hence $\ord(\psi) = n$ and $\ord^+(\psi) = m$.
\end{proof}

We present two potential approaches to deriving such schemes. Without access to the symmetric decomposition, a straightforward idea is to study the composite scheme $\psi \overline{\psi}.$ By definition, this scheme must equal the identity scheme when $\psi$ is symmetric, and by extension it is easy to see that when $\psi$ is an $EES(n,m)$ scheme, $\psi \overline{\psi}$ must match the identity up to order $m$. As such we could impose the conditions
\begin{align*}
    C(i) &:= \{\psi(\tau) = 1 / \tau!, \quad \forall |\tau| = i\}, \quad 1 \leq i \leq n,\\
    SC(i) &:= \{\psi\overline{\psi}(\tau) = 0, \quad \forall |\tau| = i\}, \quad n < i \leq m.
\end{align*}

Note that we do not need to explicitly impose $SC(i)$ for $1 \leq i \leq n$, as these follow from $C(i)$ for $1 \leq i \leq n$. It will be convenient to write $\psi\overline{\psi}(\tau) = \psi(\tilde{\tau})$, where $\tilde{\tau} := \mu \circ (\mathrm{Id} \otimes \mathrm{Inv}) \circ \Delta (\tau)$ and $\mathrm{Inv}(\tau) := \overline{\tau}$ is the canonical involution. The $SC(i)$ conditions can then be written as
\begin{align*}
    SC(i) &:= \{\psi(\tilde\tau) = 0, \quad \forall |\tau| = i\}, \quad n < i \leq m.
\end{align*}

Whilst the above is a viable approach, there is a more natural set of order conditions derived from the symmetric decomposition of B-series given in Corollary \ref{cor:RK_decomp}. Define the order conditions:
\begin{alignat*}{2}
    \mathrm{OC}(i) &= \{\psi(\tau^-) = 1/\tau!, \quad &\forall |\tau| = i\}, \quad 1 \leq i \leq n\\
    \mathrm{EC}(i) &= \{\psi(\tau^+) = 0, \quad &\forall |\tau| = i\}, \quad 1 \leq i \leq m.
\end{alignat*}

It is clear that the $C(i)$ order conditions are preferable to the more complicated $OC(i)$ conditions. However, there are several ways in which the $EC(i)$ order conditions are simpler than their counterpart, $SC(i)$. Due to Proposition \ref{prop:plus_zero}, the operation $\tau \mapsto \tau^+$ naturally generates fewer non-zero order conditions compared to $\tau \mapsto \tilde\tau$, as seen in Table \ref{table:num_conds} below. In fact, since the non-zero conditions generated by $\tau^+$ always involve the tree $\tau$, it is clear that they are independent, and as such form the smallest set of conditions possible without imposing any structural constraints on the corresponding Butcher tableau.

\begin{table}[H]
    \centering
    {\renewcommand{\arraystretch}{1.5}
  \begin{tabular}{c|ccccccccc}
    $p$ & 1 & 2 & 3 & 4 & 5 & 6 & 7 & 8 & 9 \\
    \cline{1-10}
    $\left\lvert\{\tau : \tilde\tau \neq 0, \, 1\leq|\tau| \leq p\}\right\rvert$ & 0 & 1 & 2 & 6 & 14 & 34 & 81 & 196 & 481\\
    $\left\lvert\{\tau : \tau^+ \neq 0, \, 1\leq|\tau| \leq p\}\right\rvert$ & 0 & 1 & 1 & 5 & 5 & 25 & 25 & 140 & 140
  \end{tabular}}
  \caption{Number of trees for which $\tilde\tau$ and $\tau^+$ are non-zero, forming non-trivial order conditions in $SC(i)$ and $EC(i)$ respectively.}
  \label{table:num_conds}
\end{table}

Additionally, the non-zero order conditions in $EC(i)$ are also simpler in the following sense. For an element $x \in \mathcal{H}$, define the \say{weight} of $x$, denoted $w(x)$, as the total number of nodes appearing in $x$ when it is written in its simplest form. For example,
\begin{equation*}
    w\left(\RS{nlr}\,\RS{n} + 2\,\RS{ni} - \RS{n}\,\RS{n}\right) = 8.
\end{equation*}

Given a set $\mathcal{A} \subset \mathcal{H}$ and a corresponding set of order conditions 
\begin{equation}
\label{def:orderconditions}
    A = \{\psi(x) = 0 : x \in \mathcal{A} \subset \mathcal{H}\},
\end{equation}
define the \say{weight} of $A$ by
\begin{equation*}
    w(A) := \sum_{x \in \mathcal{A}} w(x).
\end{equation*}

Note that in the construction of $EES(n,m)$ schemes for $n < m$, once we have imposed the single-tree conditions $C(i)$ for $1 \leq i \leq n$, we can substitute the values of these trees into the subsequent conditions, $SC(i)$ or $EC(i)$. This motivates the following definitions. For $x \in \mathcal{H}$, let the \say{reduced weight} $w(x, n)$ denote the weight of $x$ after all trees $\tau$ with $|\tau| \leq n$ have been replaced by $1/\tau!$. For example,
\begin{align*}
    w\left(\RS{nlr}\,\RS{n} + 2\,\RS{ni} - \RS{n}\,\RS{n}\,, 1\right) &= w\left(\RS{nlr} + 2\,\RS{ni} - 1\right) = 5,\\
    w\left(\RS{nlr}\,\RS{n} + 2\,\RS{ni} - \RS{n}\,\RS{n}\,, 2\right) &= w\left(\RS{nlr}\right) = 3.
\end{align*}

In the same way, define the \say{reduced weight} for the set of order conditions \eqref{def:orderconditions} by
\begin{equation*}
    w(A, n) := \sum_{x \in \mathcal{A}} w(x, n).
\end{equation*}

As seen in Table \ref{table:oc_weights} below, the $EC$ order conditions are more susceptible to this simplification, resulting in lower reduced weights than with the $SC$ conditions.

\begin{table}[H]
    \centering
    {\renewcommand{\arraystretch}{1.5}
  \begin{tabular}{c|cc|cc|cc}
    $i$ & $w(SC(i), 1)$ & $w(EC(i), 1)$ & $w(SC(i), 2)$ & $w(EC(i), 2)$ & $w(SC(i), 3)$ & $w(EC(i), 3)$ \\
    \cline{1-7}
    2 & 2 & 2 & - & - & - & -\\
    4 & 46 & 38 & 34 & 34 & 16 & 16\\
    6 & 807 & 665  & 641 & 549 & 503 & 387\\
    8 & 13,332 & 12,711 & 10,866 & 10,207 & 8,708 & 6,917    
  \end{tabular}}
  \caption{Reduced weights of the sets of order conditions $SC(i)$ and $EC(i)$ for various even values of $i$.}
  \label{table:oc_weights}
\end{table}

Motivated by these observations, we settle on imposing $C(i)$ for $1 \leq i \leq n$ and $EC(i)$ for $n < i \leq m$. Guided by these conditions, our aim is to derive EES schemes in the regime where $m \gg n$. Computing and imposing our order conditions is a significant task, requiring the symmetric decomposition of a large number of rooted trees. To aid with this, we use the Kauri Python package \cite{shmelev2025kauri}. We will express our schemes in terms of the general form of an explicit RK scheme,

\[
\begin{array}{c|cccccc}
0 &&&&&&\\
c_2 & a_{21} &&&&&\\
c_3 & a_{31} & a_{32} &&&&\\
c_4 & a_{41} & a_{42} & a_{43} &&&\\
\vdots & \vdots & \vdots & \vdots & \ddots &\\
c_s & a_{s1} & a_{s2} & a_{s3} & \cdots & a_{s, s-1} &\\
\hline
& b_1 & b_2 & b_3 & \cdots & b_{s-1} & b_s
\end{array}
\]

\subsection{EES(2,5)}

\begin{proposition}
\label{prop:EES_2_5}
    3-stage Runge--Kutta schemes belonging to $\mathrm{EES}(2,5)$ take the form:
    \begin{equation}
    {\renewcommand{\arraystretch}{2.2}
    \begin{array}{c|ccc}
    0 &&&\\
    \displaystyle\frac{1+2x}{4(1-x)} & \displaystyle\frac{1+2x}{4(1-x)} &&\\[7pt]
    \displaystyle\frac{3}{4(1-x)} & \displaystyle\frac{(4x-1)^2}{4(x-1)(1-4x^2)} & \displaystyle\frac{1-x}{(1-4x^2)} &\\[7pt]
    \hline
    & x & \displaystyle\frac{1}{2} & \displaystyle\frac{1}{2} - x
    \end{array}
    }
    \end{equation}
    for some $x \in \mathbb{R}$, $x \neq 1, \pm \frac{1}{2}$.
\end{proposition}

We will denote the scheme above by $\mathrm{EES}(2,5;x)$. In particular, the scheme $\mathrm{EES}(2,5; 1/4)$ is numerically simple to implement:

\begin{equation}
    {\renewcommand{\arraystretch}{1.2}
    \begin{array}{c|ccc}
    0 &&&\\
    1/2 & 1/2 &&\\
    1 & 0 & 1 &\\
    \hline
    & 1/4 & 1/2 & 1/4
    \end{array}
    }
\end{equation}

For an \say{optimal} choice of the parameter $x$, we will consider minimising the $L_1$ norm of the $C(3)$ and $EC(7)$ order conditions. Specifically, we minimise the objective function
\begin{equation*}
    \sum_{|\tau| = 3} \lvert \psi_x(\tau) - 1 / \tau! \rvert + \sum_{|\tau| = 7} \lvert \psi_x(\tau^+) \rvert \, ,
\end{equation*}
where $\psi_x$ denotes the elementary weights function of the scheme $\mathrm{EES}(2,5;x)$. Running this minimisation yields a choice of $x \approx 0.1$, corresponding to the scheme

\begin{equation}
    {\renewcommand{\arraystretch}{1.2}
    \begin{array}{c|ccc}
    0 &&&\\
    1/3 & 1/3 &&\\
    5/6 & -5/48 & 15/16 &\\
    \hline
    & 1/10 & 1/2 & 2/5
    \end{array}
    }
\end{equation}

\begin{proposition}
    $\mathrm{EES}(n,5)$ contains no 3-stage RK schemes for $n \geq 3$.
\end{proposition}

\begin{proof}
    Suppose $n \geq 3$ and $\Psi \in \mathrm{EES}(n,5)$ is a 3-stage RK scheme. As a consequence of the order conditions for $\mathrm{EES}(2,5)$, we must have
    \begin{equation*}
        \psi\,\,\RS{nIi} = 1/8,
    \end{equation*}
    which is inconsistent with the order conditions for a scheme of order $\geq 3$.
\end{proof}

\subsection{EES(2,7)}

\begin{proposition*}
   For some $x \in \mathbb{R}$, 4-stage RK schemes in $\mathrm{EES}(2,7)$ take the form:
    \begin{align*}
        b_1 &= x,\\
        b_2 &= \frac{1}{2}(2 \mp \sqrt{2}) - (1 \mp\sqrt{2})x,  & \alpha &:= \frac{(2x \pm \sqrt{2})}{(2x - 1)(-2x \mp \sqrt{2} + 1)}\\
        b_3 &= (1 \mp\sqrt{2})(x-1),  & \beta &:= \frac{1}{(2x - 1)(1 \mp \sqrt{2}-2x)(2 \mp \sqrt{2} - 2x)}\\
        b_4 &= \frac{1}{2}(2 \mp \sqrt{2}) - x,
    \end{align*}
    \begin{align*}
        a_{21} &= \frac{-2\pm \sqrt{2} (1 - 2x)}{4(x-1)}\\
        a_{31} &= \frac{(2x \pm \sqrt{2} - 2)(4x \pm \sqrt{2} - 2)}{\pm 4\sqrt{2}(x - 1)}\, \alpha\\
        a_{32} &= \frac{1}{2}(-1 \pm\sqrt{2}) \,\alpha\\
        a_{41} &= \frac{(2x \mp\sqrt{2})(-40x^4 +(80 \mp 40\sqrt{2})x^3 -(88 \mp 60\sqrt{2})x^2 +(48 \mp 34\sqrt{2})x \pm 7\sqrt{2} - 10)}{4(x-1)(2x^2-1)} \, \beta\\
       a_{42} &= (2 \mp \sqrt{2}) x (x-1) (4x \pm \sqrt{2} - 2) \, \beta\\
       a_{43} &= \frac{(2 \mp\sqrt{2})(2x \mp\sqrt{2})(2 \pm \sqrt{2} - 2x)(x-1)(2x-1)}{4(2x^2-1)(2x^2 - 4x + 1)}
    \end{align*}
\end{proposition*}

We will refer to the scheme with positive $\sqrt{2}$ and parameter $x$ as $\mathrm{EES}(2,7;x)$. In particular, $\mathrm{EES}\left(2,7;\frac{1}{4}(2 - \sqrt{2})\right)$ is numerically simple to implement:

\begin{equation}
    {\renewcommand{\arraystretch}{2}
    \begin{array}{c|cccc}
    0&&&\\
    \frac{1}{2}(2 - \sqrt{2})& \frac{1}{2}(2 - \sqrt{2})&&&\\[7pt]
    \frac{1}{2}\sqrt{2}&0&\frac{1}{2}\sqrt{2}&&\\[7pt]
     1&\frac{1}{2}(2 - \sqrt{2}) &0&\frac{1}{2}\sqrt{2}&\\[7pt]
    \hline
    &\frac{1}{4}(2 - \sqrt{2}) & \frac{1}{4}\sqrt{2}  & \frac{1}{4}\sqrt{2}   & \frac{1}{4}(2 - \sqrt{2}) 
    \end{array}
    }
\end{equation}

Similarly to the case of $\mathrm{EES}(2,5)$, we will consider minimising the $L_1$ norm of the $C(3)$ and $EC(9)$ order conditions, to find an \say{optimal} parameter $x$. Specifically, we minimise the objective function
\begin{equation*}
    \sum_{|\tau| = 3} \lvert \psi_x(\tau) - 1 / \tau! \rvert + \sum_{|\tau| = 9} \lvert \psi_x(\tau^+) \rvert \, ,
\end{equation*}

where $\psi_x$ denotes the elementary weights function of the scheme in $\mathrm{EES}(2,7;x)$. Running this minimisation yields a choice of $x \approx \frac{1}{14}(5-3\sqrt{2})$, corresponding to the scheme

\begin{equation}
    {\renewcommand{\arraystretch}{2}
    \begin{array}{c|cccc}
    0&&&\\
    \frac{1}{3}(2-\sqrt{2}) & \frac{1}{3}(2-\sqrt{2})&&&\\[7pt]
    \frac{1}{6}(2 + \sqrt{2}) & \frac{1}{24}(-4+\sqrt{2})  & \frac{1}{8}(4 + \sqrt{2})\displaystyle  &&\\[7pt]
    \frac{1}{6}(4 + \sqrt{2}) &\frac{1}{168}(-176+145\sqrt{2}) &\frac{3}{56}(8-5\sqrt{2})&\frac{3}{7}(3-\sqrt{2})&\\[7pt]
    \hline
    & \frac{1}{14}(5-3\sqrt{2}) & \frac{1}{14}(3 + \sqrt{2}) & \frac{3}{14}(-1+2\sqrt{2})   & \frac{1}{14}(9-4\sqrt{2})
    \end{array}
    }
\end{equation}

\subsection{Stability}
\label{sec:stability}

We briefly discuss the stability of the family of schemes presented above. In addition, Appendix \ref{appendix:sym_stability} provides a brief discussion of the stability of the symmetric components of RK schemes. Recall that the stability of a numerical scheme is defined in terms of its long-term behaviour when applied to linear ODEs.

\begin{definition}
    Consider applying a Runge--Kutta method $\Psi$ to the ODE $\frac{dy}{dt} = \lambda y$, $\lambda \in \mathbb{C}$. The $n^{th}$ step of the numerical approximation takes the form
    \begin{equation*}
        y_n = [R(h\lambda)]^n y_0.
    \end{equation*}
    We call $R : \mathbb{C} \to \mathbb{C}$ the stability function of the scheme $\Psi$, and the set
    \begin{equation*}
        \mathcal{D} = \{z \in \mathbb{C} : |R(z)| < 1\}
    \end{equation*}
    the stability domain. In particular, when $\{z \in \mathbb{C} : \mathrm{Re}(z) < 0\} \subset \mathcal{D}$,  $\Psi$ is said to be A-stable.
\end{definition}

Since symmetric schemes are implicit, they are often A-stable.
The EES schemes introduced above are explicit and therefore cannot be A-stable. Nonetheless, their near-symmetric properties make them surprisingly stable compared to classical explicit RK schemes, despite their low order. \cref{fig:ees} illustrates the stability regions of our 3-stage $EES(2,5; 1/10)$ and 4-stage $EES(2,7; \frac{1}{14}(5 - 3\sqrt{2}))$ schemes, compared to Kutta's 3-stage RK3, the classic 4-stage RK4 scheme and Nystr\"om's 6-stage RK5. $EES(2,5;1/10)$ achieves a stability region similar to that of RK4, despite requiring one fewer stage. Remarkably, $EES(2,7; \frac{1}{14}(5 - 3\sqrt{2}))$ is significantly more stable than both RK4 and RK5, although it only requires 4 stages. We provide detailed contour plots of the stability regions for our EES schemes, as well as their order stars \cite{wanner1978order}, in Appendix \ref{appendix:EES_stars}.

\begin{figure}[h]
    \centering

    \begin{subfigure}[t]{0.44\textwidth}
        \centering
        \includegraphics[width = \textwidth, trim={0.1cm 0.1cm 0.1cm 0.8cm},clip]{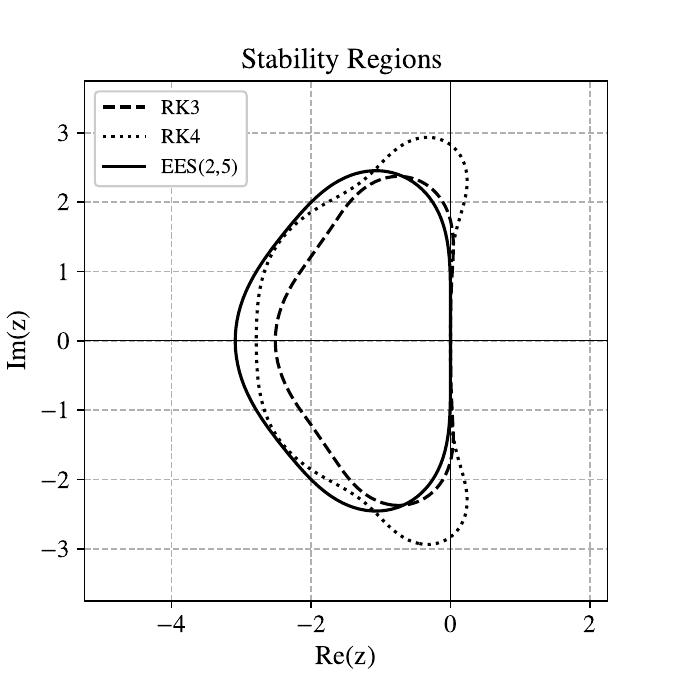}
    \end{subfigure}%
    ~ 
    \begin{subfigure}[t]{0.44\textwidth}
        \centering
        \includegraphics[width = \textwidth, trim={0.1cm 0.1cm 0.1cm 0.8cm},clip]{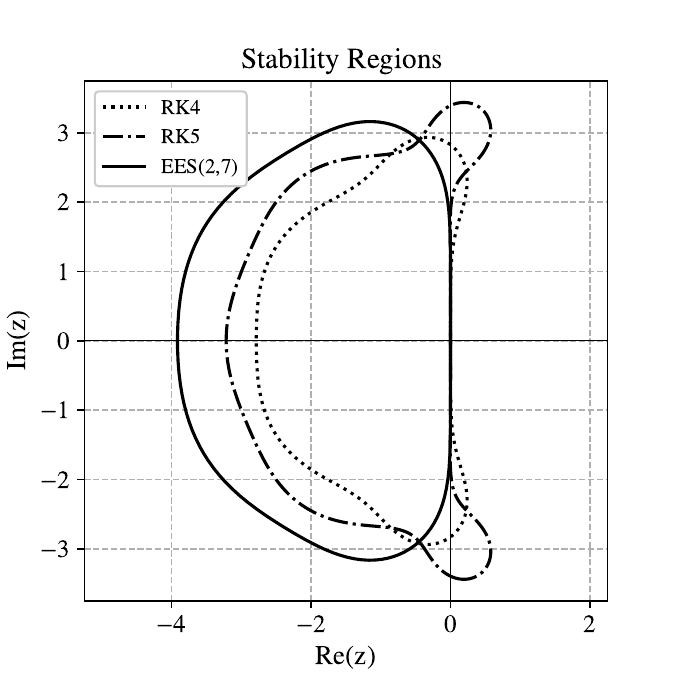}
    \end{subfigure}
    \caption{Stability domains for Kutta's RK3, RK4, Nystr{\"o}m's RK5, $EES(2,5;1/10)$ and $EES(2,7;\frac{1}{14}(5 - 3 \sqrt{2}))$}
    \label{fig:ees}
\end{figure}

\subsection{Backward Error Analysis}

Backward error analysis is a common tool in the study of geometric integrators. Whilst forward error analysis studies the local and global errors of a solution, backward error analysis aims to find a \say{modified} differential equation $\frac{d}{dt}{\tilde{y}} = f_h(\tilde{y})$ of the form \cite{hairer2006geometric}
\begin{equation}\label{eq:backward_error}
    \frac{d}{dt}\tilde{y} = f(\tilde{y}) + hf_2(\tilde{y}) + h^2f_3(\tilde{y}) + \cdots
\end{equation}

such that the exact solution to the modified equation matches the approximated solution to the original equation at all points $y_n$. The derivation of these modified equations has been extensively studied from a Hopf algebraic point of view, linking them to a new Hopf algebra on rooted trees \cite{chartier2010algebraic, calaque2011two}. When applied to symmetric schemes, it is well known that the modified equation \eqref{eq:backward_error} possesses an expansion in even powers of $h$. In this section, we briefly point out that the same is true for schemes of antisymmetric order $m$, such as our EES schemes, but only for terms up to order $h^m$.

\begin{theorem}[{\cite[Section IX.2; Theorem 2.2]{hairer2006geometric}}]
    For a symmetric method, $f_j(y)=0$ whenever $j$ is even, so that \eqref{eq:backward_error} has an expansion in even powers of $h$.
\end{theorem}

\begin{theorem}
    For a method of antisymmetric order $m$, $f_j(y)=0$ whenever $j$ is even and $j \leq m$, so that \eqref{eq:backward_error} has an expansion in even powers of $h$ up to terms of order $h^m$.
\end{theorem}

\begin{proof}
    The proof follows exactly that of \cite[Section IX.2; Theorem 2.2]{hairer2006geometric}, noting that for a method of antisymmetric order $m$, $f_j^*(y) = (-1)^{j+1}f_j(y)$ for all $j \leq m$, where $f_j^*$ are the coefficient functions of the modified equation for the adjoint method.
\end{proof}

\subsection{Numerical Experiments}

To test our EES(2,5) and EES(2,7) schemes, we take two classic examples of ODEs describing physical systems: inverse-square law attraction and galactic orbits. In the first example, we test the ability of the schemes to recover the initial condition by integrating the system forward and then reversing the direction of time. In the second, we assess their capacity to preserve the Hamiltonian of the system. We will see that, despite the low order of $\mathrm{EES}(2,7; \frac{1}{14}(5 - 3 \sqrt{2}))$, it readily outperforms high-order explicit schemes such as Nystr{\"o}m's RK5, due to its near-conservation of the Hamiltonian. As a result, we see that $\mathrm{EES}(2,7; \frac{1}{14}(5 - 3 \sqrt{2}))$ provides solutions comparable to those of implicit schemes, but with a significantly lower runtime.

\subsubsection{Inverse-Square Law Attraction}

We take from \cite{butcher2016numerical} the problem of inverse-square law attraction in two dimensions,
\begin{equation*}
    y''(x) = -\frac{1}{\lVert y(x) \rVert^{3/2}} \, y(x),
\end{equation*}
which can be rewritten as
\begin{equation}\label{eq:inverse_square_ode}
\begin{aligned}
    \frac{dy_1}{dx} &= y_3, \hspace{10mm} & \frac{dy_3}{dx} &= -\frac{y_1}{(y_1^2 + y_2^2)^{3/2}}, \\
    \frac{dy_2}{dx} &= y_4, \hspace{10mm} & \frac{dy_4}{dx} &= -\frac{y_2}{(y_1^2 + y_2^2)^{3/2}}.
\end{aligned}
\end{equation}

Taking the initial value $y(0) = [1,0,0,1]^T$, the exact solution to the ODE \eqref{eq:inverse_square_ode} is given by $y(x) = [\cos(x), \sin(x), -\sin(x), \cos(x)]$. We take $0 \leq t \leq 10$ and a step size of $h=0.1$. Table \ref{tab:inverse_square} shows the results of various EES schemes, along with implicit midpoint as a symmetric 2nd-order benchmark. We should note that the elapsed time for the implicit midpoint scheme is heavily implementation dependent and varies with the specific parameters of the root-solving algorithm used. As such, the timings are meant only as a rough indication of complexity. As expected, the EES schemes produce a low error in recovering the initial condition. The near-symmetric properties of the schemes also seem to contribute to a lower error in the solution, with $EES(2,7;\frac{1}{14}(5 - 3\sqrt{2}))$ achieving the lowest error of all the schemes.

\begin{table}[H]
    \centering
    {\renewcommand{\arraystretch}{1.5}
  \begin{tabular}{c|ccc}
    Method & Elapsed Time (s) & Error & I.C. Error\\
    \cline{1-4}
    Implicit Midpoint &0.0635 & 1.0063E-01 & 1.6939E-13\\
    \cline{1-4}
    $EES(2,5;\frac{1}{4})$ &0.0098 & 4.9232E-02 & 3.2143E-05\\
    $EES(2,5;\frac{1}{10})$ &\textbf{0.0082} & 3.0921E-02 & 7.8639E-07\\
    $EES(2,7;\frac{1}{4}(2 - \sqrt{2}))$ &0.0130 & 2.3967E-02 & \textbf{2.1530E-10}\\
    $EES(2,7;\frac{1}{14}(5 - 3\sqrt{2}))$ &0.0169 & \textbf{1.5041E-02} & 4.9545E-10
  \end{tabular}}
  \caption{Elapsed time, error in the solution and error in recovering the initial condition (I.C. error) for various Runge--Kutta schemes applied to the ODE \eqref{eq:inverse_square_ode} for $0 \leq t \leq 10$. Step size is set to $h = 0.1$ for all schemes, and the implicit solver is limited to 100 iterations for implicit midpoint.}
  \label{tab:inverse_square}
\end{table}

\subsubsection{Galactic Orbits}

To test EES schemes on Hamiltonian mechanics, we take the example from \cite[Section II.16]{hairer1993nonstiff} of the orbit of one of $N$ stars in a galaxy, in the potential formed by the other $N-1$. We present the resulting mechanics without details of their derivation, and instead refer the interested reader to \cite[Chapter 3]{binney2011galactic}. Take the Hamiltonian
\begin{align*}
    H &= \frac{1}{2}(p_1^2 + p_2^2 + p_3^2) + \Omega (p_1 q_2 - p_2 q_1) + A \ln\left(C + \frac{q_1^2}{a^2} + \frac{q_2^2}{b^2} + \frac{q_3^2}{c^2}\right),\\[5mm]
    a&= 1.25, \quad b = 1, \quad c = 0.75, \quad A = 1, \quad C = 1, \quad \Omega = 0.25,\\
    q_1(0) &= 2.5, \quad q_2(0) = 0, \quad q_3(0) = 0, \quad p_1(0) = 0, \quad p_3(0) = 0.2,
\end{align*}

and set $p_2(0) \approx 1.69$ to be the larger of the roots for which $H = 2$. As in \cite{hairer1993nonstiff}, we compute the orbit for $0 \leq t \leq 1,000,000$ and consider the Poincar\'e sections for $q_2 = 0$, $q_1 > 0$, $\dot{q}_2 > 0$. We consider different explicit methods with a step size of 1/40, and compare the mean average error (MAE) of the Hamiltonian, as well as the number of points in the Poincar\'e section. For an \say{exact} solution, we use the DOP853 solver with a tolerance of $10^{-15}$.\par\medskip

Table \ref{tab:galactic} shows the results for various classic explicit RK methods, as well as  $EES(2,7;\frac{1}{14}(5 - 3 \sqrt{2}))$. Despite its low order, the EES scheme readily outperforms classic high-order schemes such as RK4 and Nystr{\"o}m's RK5. Remarkably, it produces a Poincar\'e section containing only 2 more points than the exact solution, and yields a Hamiltonian with mean absolute error (MAE) $8.96 \times 10^{-12}$, roughly 1853 times lower than Nystr{\"o}m's RK5. We plot the Poincar\'e section only for $EES(2,7;\frac{1}{14}(5 - 3 \sqrt{2}))$ and the exact solution, and refer the reader to \cite[Section II.16]{hairer1993nonstiff} for detailed plots corresponding to other explicit and implicit methods for comparison. Figure \ref{fig:Poincare_plots} shows these two sections. Once again, the Poincar\'e section produced by $EES(2,7;\frac{1}{14}(5 - 3 \sqrt{2}))$ is remarkably similar to the exact solution, despite the scheme's low order.

\begin{table}[H]
    \centering
    {\renewcommand{\arraystretch}{1.5}
  \begin{tabular}{c|ccc}
    Method & Order & \# Poincar\'e Points & Hamiltonian MAE\\
    \cline{1-4}
    Exact & - & 47101 & -\\
    \cline{1-4}
    Heun's RK2 & 2 & 9455 & 1.37E-3\\
    Heun's RK3 & 3 & 47766 & 4.90E-4\\
    Classic RK4 & 4 & 47004 & 1.08E-7\\
    Ralston's RK4 & 4 & 46991 & 1.18E-7\\
    Nystr{\"o}m's RK5 & 5 & 47132 & 1.66E-8\\
    $EES(2,7;\frac{1}{14}(5 - 3 \sqrt{2}))$ & 2 & \textbf{47103} & \textbf{8.96E-12}
  \end{tabular}}
  \caption{Properties of the solution trajectories for various explicit schemes with a step size of $h = 1/40$.}
  \label{tab:galactic}
\end{table}

\begin{figure*}[h]
    \centering
    \includegraphics[width = 0.7\textwidth]{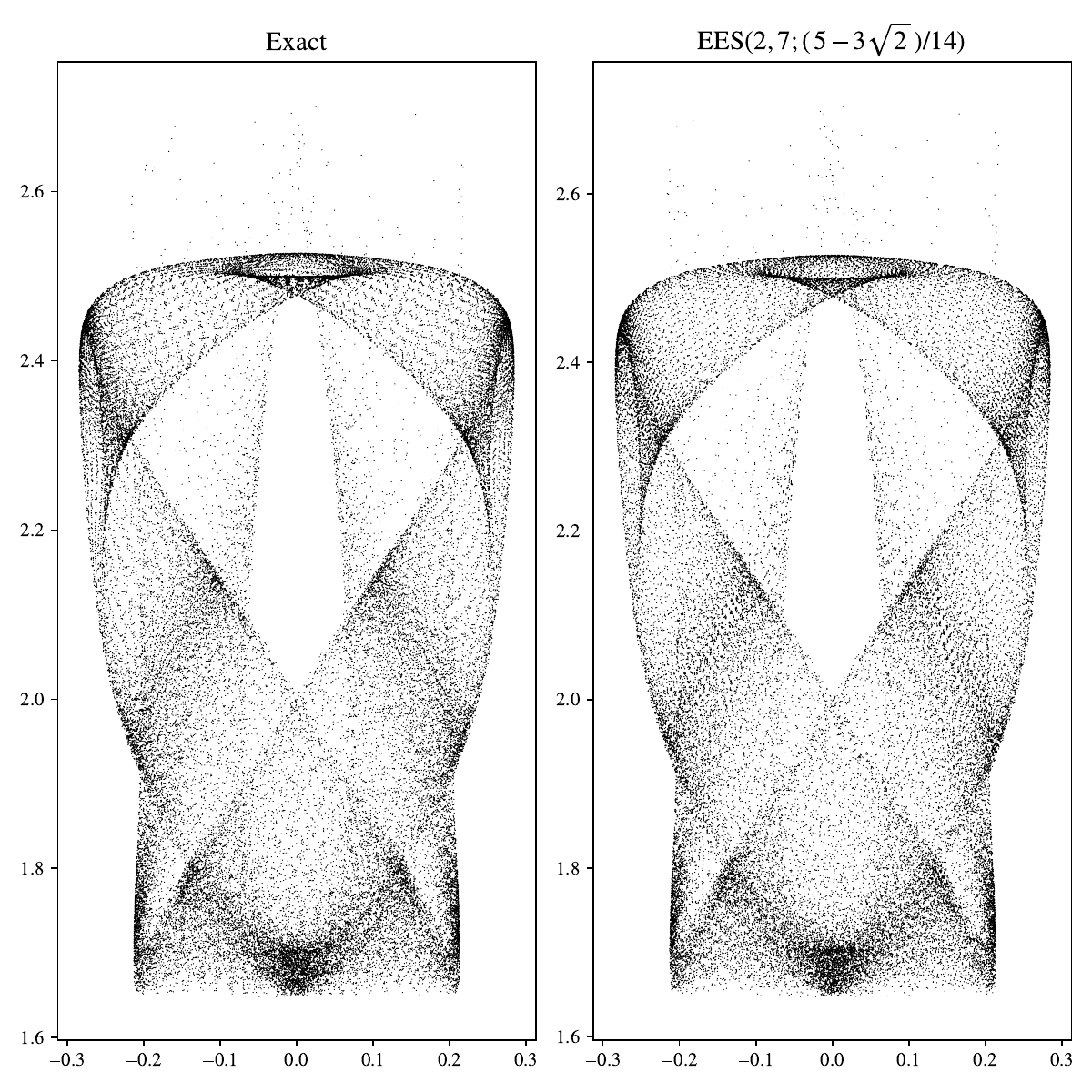}
    \caption{Poincar\'e sections for the exact solution (left) and the solution given by $EES(2,7;\frac{1}{14}(5 - 3 \sqrt{2}))$ (right).}
    \label{fig:Poincare_plots}
\end{figure*}

\section{Conclusion}
\label{sec:conclusions}

In this paper, we have discussed the Hopf-algebraic structure of symmetric schemes and introduced the symmetric decomposition of B-series methods as a tool for analysing the symmetry of B-series and Runge--Kutta methods. Motivated by this decomposition, we introduced a new class of Explicit and Effectively Symmetric (EES) Runge--Kutta schemes, which exhibit near-symmetric properties at a significantly lower computational cost than classical implicit symmetric methods. We demonstrated the performance of these schemes on two initial value problems, and showed how the schemes produce superior results when symmetry and energy conservation are of particular importance.\par\medskip

There are several potential extensions of EES schemes that are left for future research. In addition to classical applications, EES schemes may prove useful for training neural ODEs \cite{chen2018neural}, where low-order symmetric schemes are required for training. Other models such as Neural SDEs \cite{kidger2021neural} and Neural Jump ODEs \cite{herrera2020neural} have been of interest in areas including finance, with a limited number of symmetric schemes available for their training \cite{kidger2021efficient}. An extension of this paper's analysis to the stochastic case would provide a new class of efficient schemes for the training of these models. In addition, extensions to partitioned symmetric schemes \cite{wandelt2012symmetric} or schemes with adaptive step sizes may provide more robust methods for stiff equations. For cases where a higher antisymmetric order is necessary, a derivation of $\mathrm{EES}(n, 9)$ schemes may be of use.

\bibliographystyle{siamplain}
\bibliography{references}

\appendix

\section{Stability of Symmetric Components} \label{appendix:sym_stability}

In this appendix, we will briefly discuss the stability of the symmetric components of RK schemes. As we will show below, given a RK scheme $\Psi$, its symmetric component $\Psi^-$ is not in general a RK scheme itself. It is therefore easier to analyse the stability of the compositional square of the symmetric component, which can be written as $(\Psi^-)^{\circ 2} = \Psi\circ\Psi^*$ and is clearly a RK scheme itself. Recall that the stability domain for a RK scheme is given by \cite{wanner1996solving}
\begin{equation} \label{eq:RK_stability_domain}
    \mathcal{D} = \left\{z \in \mathbb{C} : \left\lvert \frac{\det(I - zA + zeb^T)}{\det(I-zA)} \right\rvert < 1 \right\},
\end{equation}
where $e^T=(1,\ldots,1)$. Thus, the stability function of a RK scheme is a rational function which is the quotient of two characteristic polynomials. It is easy to see that the composition of two RK schemes has a stability function given by the product of stability functions of the individual schemes. A simple consequence is the following.

\begin{proposition}
    Let $\Psi$ be a Runge--Kutta scheme with stability function $R : \mathbb{C} \to \mathbb{C}$. Then the scheme given by $\Psi \circ \Psi^*$ has stability function
    \begin{equation*}
        \widetilde R(z) = \frac{R(z)}{R(-z)}.
    \end{equation*}
\end{proposition}

\begin{proof}
The result follows immediately from $\Psi^*_h = \Psi_{-h}^{-1}$ and the definition of the stability function.
\end{proof}

To see that the symmetric component is not in general a RK scheme, suppose by contradiction that the character $\psi^- = (\psi^*\psi)^{1/2}$ corresponds to an RK scheme $\Omega$. Since $\Omega \circ \Omega = \Psi \circ \Psi^*$, it must be the case that $\Omega$ has the stability function $(R(z) / R(-z))^{1/2}$. But this is not a rational function in general, and so cannot be the stability function of a RK scheme. For example, let $\Psi$ be the Euler scheme, and consider its square symmetric component $\psi^* \psi = (\psi^-)^2$. This corresponds to an RK scheme with stability function
\begin{equation*}
    R(z) = \frac{1+z}{1-z}.
\end{equation*}
But $\sqrt{1+z} / \sqrt{1-z}$ is not a rational function, and so $\psi^-$ cannot correspond to an RK scheme. For the remainder of this section, we will center our analysis around the squared symmetric component $(\Psi^-)^{\circ 2} = \Psi \circ \Psi^*$, since this is guaranteed to be a RK scheme and has the same stability domain as the B-series method $\Psi^-$. As with RK schemes, we define the stability domain of a B-series method to be the set of $\lambda \in \mathbb{C}$ for which the method applied to $dy/dt = \lambda y$ decays to $0$ as $t \to \infty$.

\begin{proposition}
    Let $\Psi$ be an explicit Runge--Kutta  method with stability function $R$. Then $\Psi^-$ is A-stable if and only if the zeros of $R(z)$ have negative real parts.
\end{proposition}

\begin{proof}
    Since $\Psi$ is explicit, $\det(I - zA) = 1$ and $R(z)$ is a polynomial. Moreover, it is easy to see that $R$ must have real coefficients, since
    \begin{align*}
        R(z) &= \det(I - zA + zeb^T)\\
        &= z^n \det\left(\frac{1}{z}I - A + eb^T\right)\\
        &= z^n \chi(1/z)
    \end{align*}
    where $n$ is the degree of $R$ and $\chi$ is the characteristic polynomial of $A - eb^T$. Let $\widetilde R(z) = R(z) / R(-z)$. Recall that $\widetilde R$ is A-acceptable if and only if its poles have positive real parts and $|\widetilde R(it)| \leq 1$ for all $t \in \mathbb{R}$. Note that the latter condition is satisfied trivially since
    \begin{align*}
        |\widetilde R(it)| \leq 1 \quad \forall t &\Leftrightarrow |R(it)|^2 \leq |R(-it)|^2 \quad \forall t\\
        &\Leftrightarrow R(it)\overline{R(it)} \leq R(-it)\overline{R(-it)} \quad \forall t\\
        &\Leftrightarrow R(it)R(-it) \leq R(-it)R(it) \quad \forall t
    \end{align*}
    using the fact that $R$ has real coefficients. The result follows for $\Psi \circ \Psi^*$ and hence also for $\Psi^-$.
\end{proof}

\begin{example}
    For the Euler scheme,
    \begin{equation*}
        R(z) = 1 + z
    \end{equation*}
    has only one zero at $z=-1$. Thus the resulting symmetric component is A-stable, as depicted in \cref{fig:euler}. It can easily be verified that the squared symmetric component of the Euler method is the Crank--Nicolson method.

    \begin{figure}[H]
    \centering
    \includegraphics[width = \textwidth]{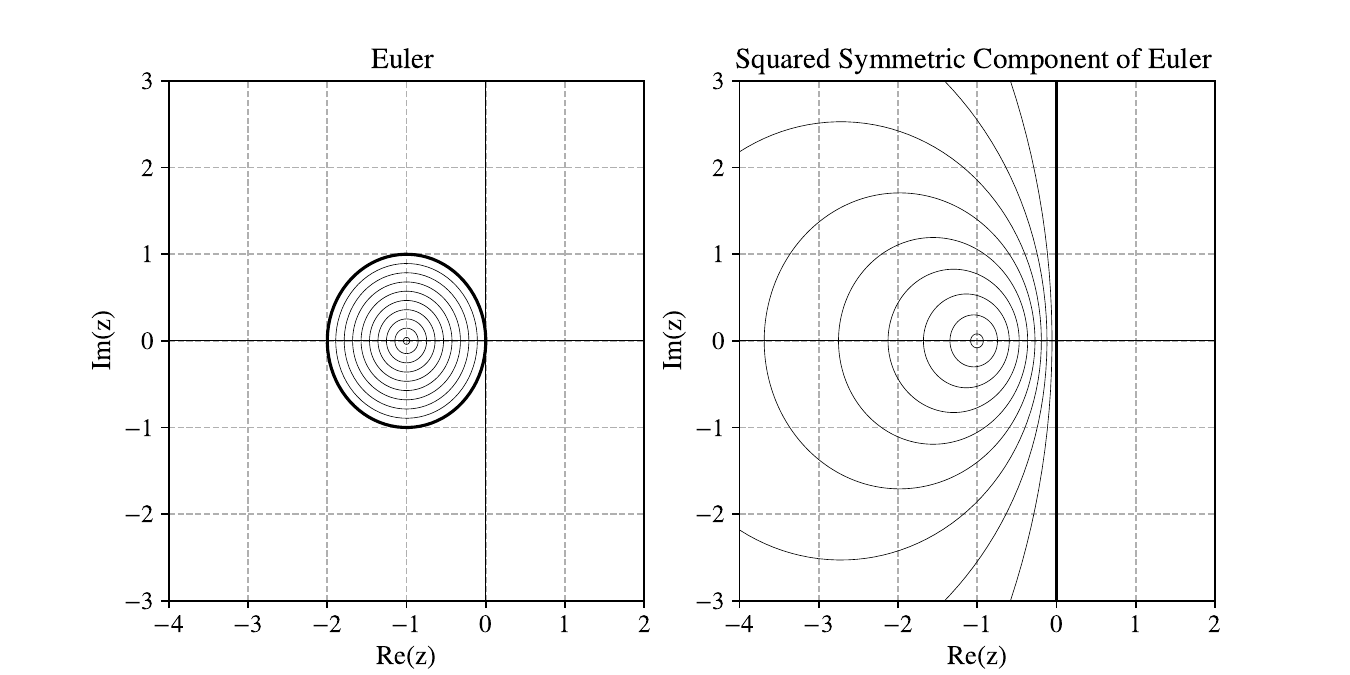}
    \caption{Stability domain for Euler and the corresponding squared symmetric component $(\Psi^-)^{\circ 2}$.}
    \label{fig:euler}
    \end{figure}
\end{example}

\begin{example}
    For the implicit midpoint scheme, the squared symmetric component is the symplectic diagonally implicit RK method:
    \begin{equation*}
        \begin{array}{c|cc}
        1/2 & 1/2 & 0\\
        3/2 & 1 & 1/2\\
        \hline
        & 1/2 & 1/2
        \end{array}
    \end{equation*}
\end{example}

\begin{example}
    For RK4,
    \begin{equation*}
        R(z) = 1 + z + \frac{z^2}{2!} + \frac{z^3}{3!} + \frac{z^4}{4!}
    \end{equation*}
    and one can show that the zeros of $R(z)$ have negative real parts. Thus the symmetric component is A-stable as shown in \cref{fig:rk4}.

    \begin{figure}[H]
    \centering
    \includegraphics[width = \textwidth]{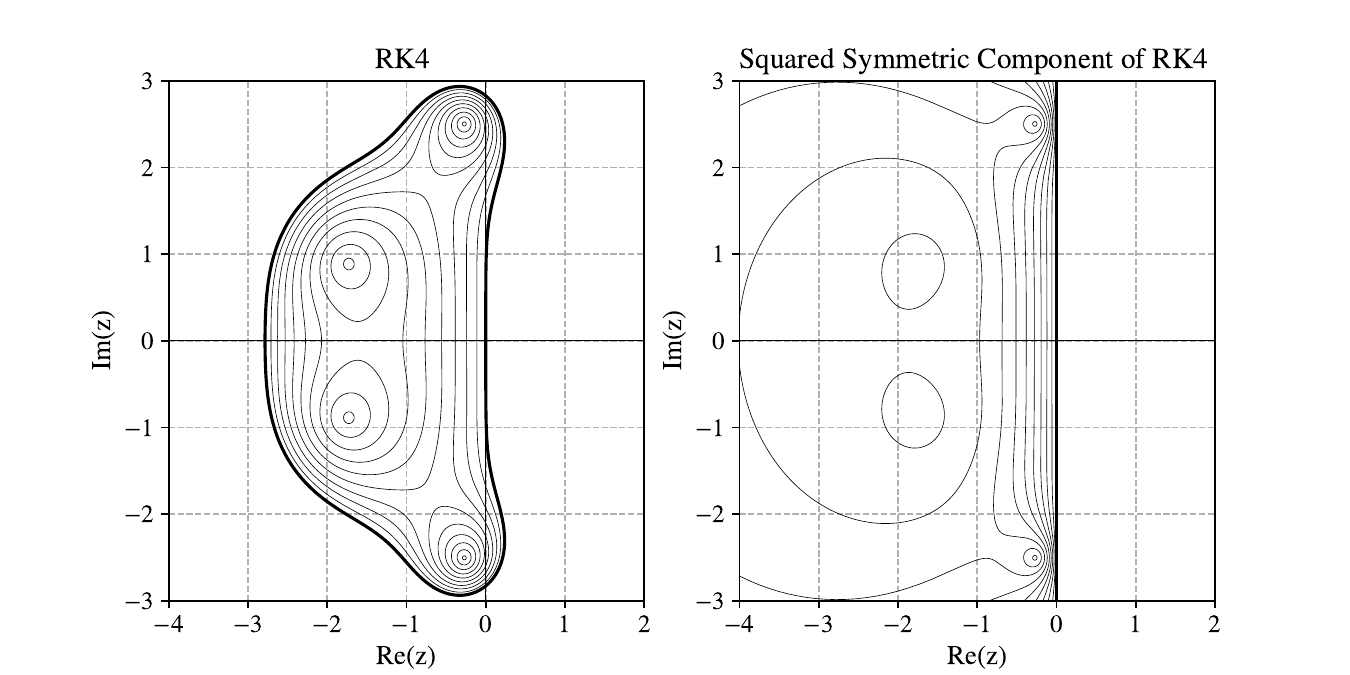}
    \caption{Stability domain for RK4 and the corresponding squared symmetric component $(\Psi^-)^{\circ 2}$.}
    \label{fig:rk4}
    \end{figure}
\end{example}

\section{Stability and Order Stars of EES Schemes} \label{appendix:EES_stars}

In this section, we provide detailed contour plots of the stability functions of $EES(2,5;1/10)$ and $EES(2,7;(5 - 3\sqrt{2})/14)$, as well as their order stars. Given a RK scheme with stability function $R(z)$, the order star \cite{wanner1978order} is defined as the region
\begin{equation*}
    \{z \in \mathbb{C} : |R(z)| > |e^z| \}.
\end{equation*}

\begin{figure}[H]
    \centering
    \includegraphics[width = \textwidth]{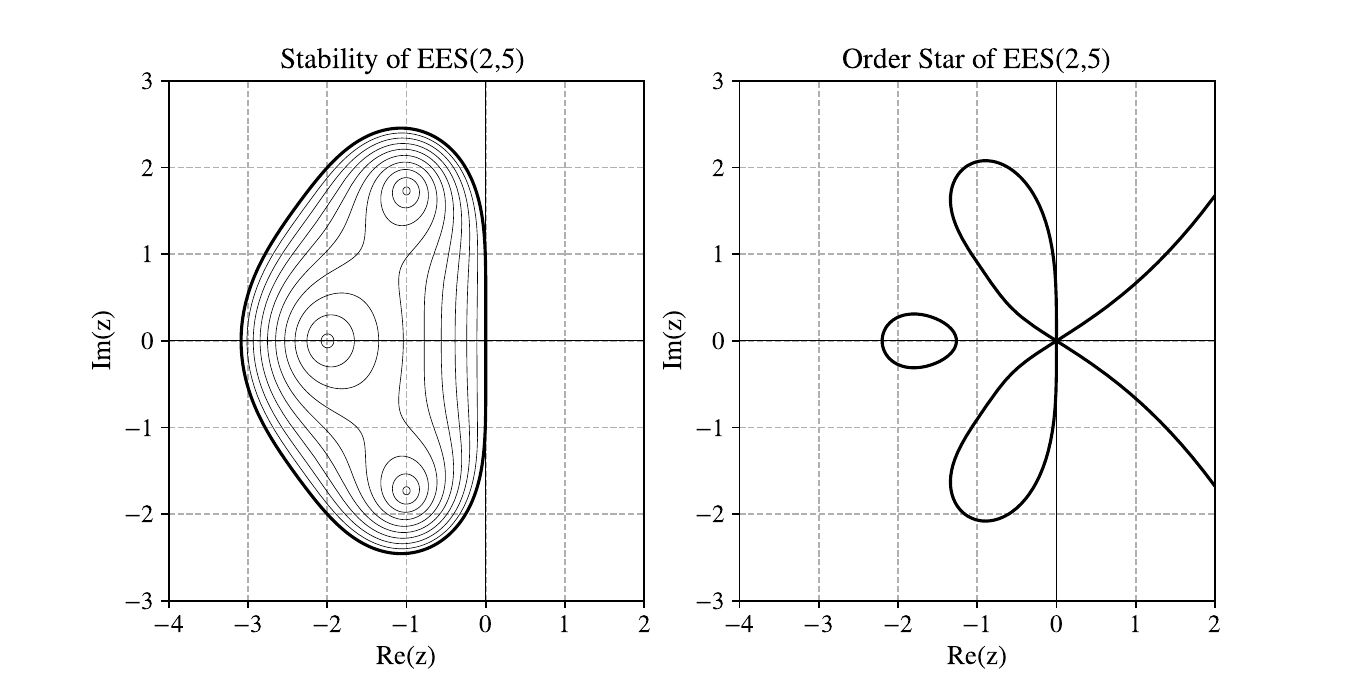}
    \caption{Stability domain and order star for $EES(2,5;1/10)$}
    \label{fig:ees25_star}
\end{figure}

\begin{figure}[H]
    \centering
    \includegraphics[width = \textwidth]{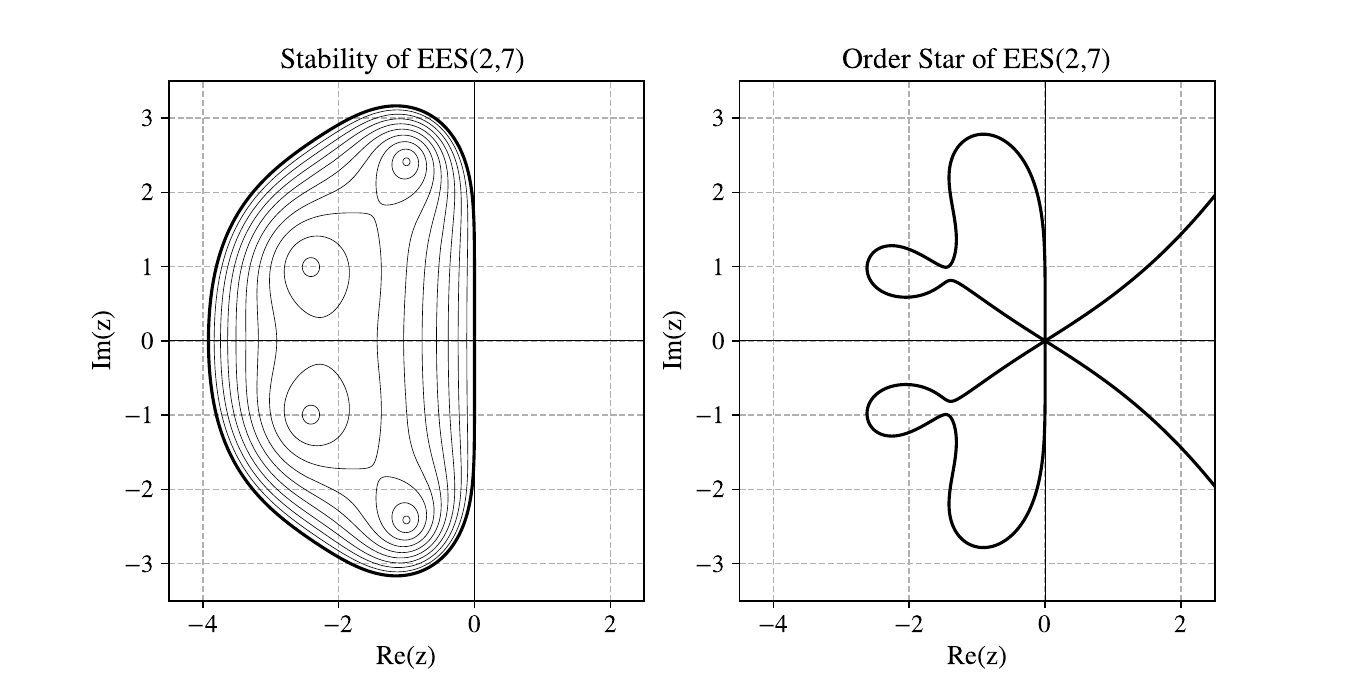}
    \caption{Stability domain and order star for $EES(2,7;(5 - 3\sqrt{2})/14)$}
    \label{fig:ees27_star}
\end{figure}

\end{document}